\documentclass[fontsize=12pt,a4paper,headings=normal,
twoside=false,leqno,parskip=half-,abstract=true]{scrartcl}
\usepackage[english]{babel}
\usepackage[utf8]{inputenc}
\usepackage{hyperref}
\hypersetup{
 pdftitle={Sturm attractors for fully nonlinear parabolic equations},
 pdfauthor={Phillipo Lappicy},
 colorlinks=true,
 linkcolor=blue,
 citecolor=blue,
 filecolor=blue,
 urlcolor=blue}

\usepackage{graphicx}
\usepackage[format=plain,labelfont=bf,font=small]{caption}
\usepackage{subfigure}%replaced by subcaption
\usepackage{xcolor}
\usepackage[arrow, matrix, curve]{xy}
\usepackage{tikz}
\usepackage{float}

\usepackage{tikz-cd}
\usepackage{caption}
\usepackage{bm}

\usepackage{tabulary}
\usepackage{array}
\newcolumntype{N}[1]{>{\centering\arraybackslash}m{#1}}

\usepackage{amsmath,amsthm}
\usepackage{amssymb} 
\makeatletter
\newcommand{\tpitchfork}{%
  \vbox{
    \baselineskip\z@skip
    \lineskip-.52ex
    \lineskiplimit\maxdimen
    \m@th
    \ialign{##\crcr\hidewidth\smash{$-$}\hidewidth\crcr$\pitchfork$\crcr}
  }%
}
\makeatother
\usepackage{latexsym}

\usepackage[notref,notcite,color,final %make final instead of draft, by % %
]{showkeys}
\usepackage[pagewise]{lineno}%\linenumbers %% for numbering of lines
\definecolor{refkey}{rgb}{1,0,0}
\definecolor{labelkey}{rgb}{1,0,0}

\usepackage{tikz}

\usepackage[textwidth=2cm,textsize=small,backgroundcolor=none]{todonotes}

  \mathchardef\ordinarycolon\mathcode`\:
  \mathcode`\:=\string"8000
  \begingroup \catcode`\:=\active
    \gdef:{\mathrel{\mathop\ordinarycolon}}
  \endgroup

\newtheorem{thm}{Theorem}[section]
\newtheorem{lem}[thm]{Lemma}
\newtheorem{prop}[thm]{Proposition}

\usepackage{comment}
\usepackage{caption}

\begin{document}

\title{{\LARGE{Sturm attractors for \\ fully nonlinear parabolic equations}}}

\author{
 \\
{~}\\
Phillipo Lappicy*\\
\vspace{2cm}}

\date{ }
\maketitle
\thispagestyle{empty}

\vfill

%$\ast$ \\
%Institut für Mathematik\\
%Freie Universität Berlin\\
%Arnimallee 3\\ 
%14195 Berlin, Germany\\
%\\
$\ast$\\
Instituto de Ciências Matemáticas e de Computação, Universidade de S\~ao Paulo\\
Av. trabalhador são-carlense 400, 13566-590, São Carlos, SP, Brazil\\
$\ast$\\
Instituto Superior T\'ecnico, Universidade de Lisboa\\
Av. Rovisco Pais, 1049-001 Lisboa, Portugal\\
%{\tt lappicy@hotmail.com}\\
%%%%%%%%%%%%%%%%%%%%%%%%%%%%%%%%%%%%%%%%%%%%%%%%%%%%%%%%%%%

\newpage
\pagestyle{plain}
\pagenumbering{arabic}
\setcounter{page}{1}

\begin{abstract}
\noindent
We explicitly construct global attractors of fully nonlinear parabolic equations in one spatial dimension.  
These attractors are decomposed as equilibria (time independent solutions) and heteroclinic orbits (solutions that converge to distinct equilibria backwards and forwards in time).
In particular, we state necessary and sufficient conditions for the occurrence of heteroclinics between hyperbolic equilibria, which is accompanied by a method that computes such conditions. %Lastly, we compute a fully nonlinear version of the Chafee-Infante attractor.

\ 


%MSC2020: 35K55, 35B40, 35B41, 37L30, 37G35, 37L45
\end{abstract}

\section{Main results}\label{sec:intro}

\numberwithin{equation}{section}
\numberwithin{figure}{section}
\numberwithin{table}{section}

Consider the scalar fully nonlinear parabolic differential equation
\begin{equation}\label{PDE}
    f(x,u,u_x,u_{xx},u_t)=0,
\end{equation}
with initial data $u(0,x)=u_0(x)$, where $x\in (0,\pi)$ has Neumann boundary conditions, $f \in C^2$. Indices abbreviate partial derivatives.
We assume the parabolicity condition 
\begin{equation}\label{par}     %\tag{par}
    f_q\cdot f_r<0,
\end{equation}
for every argument $(x,u,p,q,r):=(x,u,u_x,u_{xx},u_t)$. In particular, $f_q,f_r\neq 0$.
To present the remaining hypothesis, we rewrite \eqref{PDE} in two different manners.%, following \cite{LappicyFiedler18}.

On one hand, we solve for the diffusion variable $q=u_{xx}$ in terms of the other variables $(x,u,u_x,u_t)$. Indeed, the parabolicity condition \eqref{par} implies $f_q\neq 0$, which allows us to rewrite \eqref{PDE} by the implicit function theorem as
\begin{equation}\label{diffIFT}
    u_{xx}=F(x,u,u_x,u_t),
\end{equation}
where the parabolicity condition \eqref{par} becomes $F_r>0$ at any $(x,u,p,r)=(x,u,u_x,u_t)$, since implicit differentiation implies $F_r=-f_r/f_q>0$. We note that $F$ may not, and need not be, defined globally. 
We only consider $F(x,u,p,.)$ to be defined on an open interval of $r$, with limits $\pm \infty$ of $F$ at the boundaries.

Then, we split the function $F$ into two parts: one independent of $u_t$, and the other depending on $u_t$. 
In other words, we distinguish between the diffusion part $F^0$ related to the equilibrium ODE, $u_t=0$; and the diffusion part $F^1$ related to time changing solutions. 
Specifically, to account for equilibria, we define
\begin{equation}\label{F0}
    F^0(x,u,p):=F(x,u,p,0),   
\end{equation}
and suppose it is well-defined. 
Otherwise, let $F^0\equiv0$, artificially.
%We do not comment further on this uninteresting case, below.
%{\color{blue} PS: Note $F(x,u,p,.)$ is defined on an open interval $I$ of $r$. If $0\not\in I$, then $F^0$ is not well defined. Should we define $F^0\equiv 0$ in such case?}
To distinguish the $u_t$ dependence, define
\begin{equation}\label{F1}
    F^1(x,u,p,r):=
    \begin{cases}
        \frac{F(x,u,p,r)-F^0(x,u,p)}{r}\, &\text{ for } r\neq 0\\
        F_r(x,u,p,0)\, & \text{ for } r= 0.
    \end{cases}
\end{equation}
The parabolic equation \eqref{diffIFT} can be rewritten as 
\begin{equation}\label{diffIFTsplit}
    u_{xx}=F^0(x,u,u_x)+F^1(x,u,u_x,u_t)u_t\,.
\end{equation}
The parabolicity condition \eqref{par}, or $F_r>0$, now reads $F^1>0$. Indeed, the monotonicity condition $F_r>0$ ensures $F^1>0$ at $r=0$, as well as $\textrm{sign}(r)=\textrm{sign}(F-F^0)$ for $r\neq 0$. In the latter case, the numerator and denominator in \eqref{F1} have the same sign, yielding $F^1>0$ for all $r$.
%Note equation \eqref{diffIFTsplit} is almost the quasilinear equation examined in \cite{Lappicy18}, except here the nonlinear term $F^1$ is allowed to be dependent on $u_{t}$.

On the other hand, we solve for the evolution variable $r=u_{t}$ in terms of the others $(x,u,u_x,u_{xx})$. Again, the implicit function theorem allows us to rewrite \eqref{PDE} as 
\begin{equation}\label{diffIFT2}
    u_t=\tilde{F}(x,u,u_x,u_{xx}).%=\tilde{F}^0(x,u,u_x)+\tilde{F}^1(x,u,u_x,u_{xx})u_{xx}.  
\end{equation}
We split $F$ analogously to \eqref{diffIFTsplit}, i.e.,
\begin{equation}\label{diffIFT2split}
    u_t=\tilde{F}^0(x,u,u_x)+\tilde{F}^1(x,u,u_x,u_{xx})u_{xx}  
\end{equation}
where $\tilde{F}^0,\tilde{F}^1$ are defined similarly to \eqref{F0},\eqref{F1}.
The parabolicity condition is $\tilde{F}^1>0$. %Again, this is similar the quasilinear equation studied in \cite{Lappicy18}, except here the nonlinear term $\tilde{F}^1$ is allowed to be dependent on $u_{xx}$.

%\textcolor{red}{PL: Maybe give an example of the above, where the splittings yield different equations at the end of the day? also that yield different equations than the quasilinear equations I did in my thesis... like any $F_1(u_t)>0$ and $\tilde{F}^1>0$. %Maybe I should mention the degenerate equations, when the parabolicity conditions fail? Indeed, in the discussion session I should give an overview on the problems to be tackled: nonautonomous, degenerate equations, hyperbolic equations, nonlocal (both in reaction and diffusion), ... 
%}

%\textcolor{red}{PL: are there ways to solve by IFT two different branches of the equation so that they have different attractors on each branch? For example, see the nonlinear diffusion being $(|u_{xx}|^{p-2}+1)u_{xx}$ in the following link: %https://onlinelibrary.wiley.com/doi/full/10.1002/mma.5957?casa_token=Y-iyGD9nfckAAAAA%3Ab iPWjGYkzuK8KkKQoeyZde4VmdIW09feyVZag6AbqxWnpUFP48CTz9EnbdfTvM-YD1PeiG2Oi3X27IH 
%How do we solve for $u_{xx}$ here? firstly, is it even s.t. $f_q\neq 0$
%}

Therefore equation \eqref{PDE} can be disguised as \eqref{diffIFTsplit} or \eqref{diffIFT2split}, and we use each splitting whenever it is more convenient. For instance, the former splitting  \eqref{diffIFTsplit} is nonlinear in $u_t$, but linear in $u_{xx}$; %, which suits better for the shooting methods in Section \ref{sec:perm}; 
whereas the latter \eqref{diffIFT2split} is linear in $u_t$, but nonlinear in $u_{xx}$.%, which is more compatible with the liberalism proof in Section \ref{sec:globalsturm}. 

The equation \eqref{diffIFT2split} defines a local \emph{semiflow} denoted by $(t,u_0)\mapsto u(t)$ in a Banach space $X^\alpha:=C^{2\alpha+\beta}([0,\pi])\cap \{\text{Neumann b.c.}\}$ for some $\alpha,\beta\in (0,1)$.
Consider $2\alpha+\beta>1$ so that solutions are at least $C^1$. The appropriate functional setting is described in Section \ref{sec:func}. 

We suppose that the semiflow is \emph{bounded} and \emph{dissipative}: trajectories $u(t)$ remain bounded, exist for all time $t\in \mathbb{R}_+$ and eventually enter a large ball in the phase-space $X^\alpha$. %See Chapter 6, Section 5 in \cite{Ladyzhenskaya68}. Also \cite{KruzhkovOleinek61} and \cite{BabinVishik92}. %ADD ALSO BOUNDEDNESS OF 2nd DERIVATIVE AS IN http://link.springer.com/article/10.1007/BF02433437
These hypotheses guarantee the existence of a nonempty \emph{global attractor} $\mathcal{A}\subseteq X^\alpha$ of \eqref{PDE}, which is the maximal compact invariant set. Equivalently, it is the minimal set that attracts all bounded sets of $X^\alpha$, or the set of all bounded trajectories $u(t)$ in $X^\alpha$ that exist for all $t\in \mathbb{R}$. See Figure \ref{figAtt}. The abstract setting is accurately described in the monumental work of Uraltseva, Ladyzhenskaya and Solonnikov \cite{Ladyzhenskaya68} for quasilinear equations, Henry \cite{Henry81} for semilinear equations, Hale, Magalhaes and Oliva \cite{HaleMagalhaesOliva84} for a general abstract framework, Lunardi \cite{Lunardi91} for fully nonlinear equations, Babin and Vishik \cite{BabinVishik92} for quasilinear equations.

%In order to study the long time behavior of \eqref{PDE},
To obtain boundedness and dissipativity of the semiflow,
we suppose that $F$ and $\tilde{F}$ can be globally solved yielding global descriptions of the PDE \eqref{PDE} as \eqref{diffIFTsplit} and \eqref{diffIFT2split}, and $\tilde{F}$ satisfies the following growth conditions: 
\begin{subequations}\label{diss}
\begin{align}
    \tilde{F}_q(x,u,p,q) & \geq \nu(|u|), &\text{ for all $(x,u,p,q)$,} \\ 
    \tilde{F}(x,u,0,0)\cdot u & \leq c \cdot [1+u^2], &\text{ for all $(x,u)$,}\\ 
    |\tilde{F}_x(x,u,p,0)| & \leq \mu_2(|u|)[1+|p|], &\text{ for all $(x,u,p)$,}\\
    \tilde{F}_u(x,u,p,0) & \leq \mu_1(|u|), &\text{ for all $(x,u,p)$,}\\
    \tilde{F}(0,0,p,0) \leq & \,\, 0 \leq \tilde{F}(1,0,p,0), & \text{ for all $|p|>K$ and some $K\in\mathbb{R}$,}
    %\text{sign} (\tilde{F}(0,0,p,0)) & \neq \text{sign} (\tilde{F}(1,0,p,0)) &
\end{align}
\end{subequations}
where $\nu$ is positive and continuously decreasing; and $\mu_1$ is continuously increasing. See Proposition 3.5 in \cite{Lunardi91}, or Chapter 6, Section 5 in \cite{Ladyzhenskaya68}. Also \cite{KruzhkovOleinek61, BabinVishik92}. 
We also need that the fractional power $\alpha>1/2$.
Note that the choice $\alpha\in (1/2,1)$ guarantees that $2\alpha+\beta>1$ for any $\beta\in (0,1)$.

\begin{figure}[H]
\centering
\begin{tikzpicture}[scale=1]
    \draw [fill=lightgray!50,domain=0:6.28,variable=\t,smooth] plot ({sin(\t r)},{0.7*cos(\t r)}) node[anchor= north east] {\footnotesize{$\mathcal{A}$}}; 
    
    \draw (-1.5,0) -- (1.5,0);
    \draw (0,-1.5) -- (0,1.5) node[right] {\footnotesize{$X^\alpha$}};
    
    \filldraw [black] (-0.75,-1.125) circle (1pt) node[anchor=north]{\footnotesize{$u_0$}};

    \draw [->,domain=-0.5:0.5,variable=\t,smooth] plot ({\t-0.25},{(\t)^3-1}) node[right] {\footnotesize{$u(t)$}}; 
\end{tikzpicture}
\caption{The semiflow $u(t)$ is a time $t\in\mathbb{R}_+$ action of the space $X^\alpha$ such that orbits are time-parametrized curves in $X^\alpha$ displaying the time evolution of an initial point $u_0\in X^\alpha$. Moreover, any initial data converges to the global attractor $\mathcal{A}$.}\label{figAtt}
\end{figure}
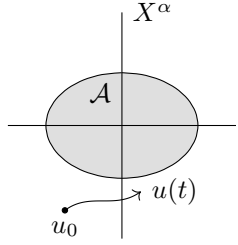

Next, we seek to decompose the attractor into smaller invariant sets, and describe how those sets are related within its internal dynamics. 
%This conjecture to construct the global attractor for fully nonlinear equations was originally stated by Fiedler in \cite{Fiedler96}.
The attractor $\mathcal{A}\subseteq X^\alpha$ has gradient dynamics, due to the existence of a Lyapunov function in \cite{LappicyFiedler18}, which generalizes the construction of Matano \cite{Matano88} for quasilinear equations. See also Zelenyak \cite{Zelenyak68}.
More precisely, there is an energy functional $E:= \int_{0}^\pi L(x,u,u_x) dx$ that decays according to
\begin{equation}\label{BOOK}
        \frac{dE}{dt}:= -\int_{0}^\pi L_{pp}(x,u,u_x)F^1(x,u,u_x,u_t) |u_t|^2 dx\leq 0,
\end{equation}
for some Lagrange function $L$ satisfying the convexity condition $L_{pp}>0$, where $p:=u_x$, and the positive function $F^1>0$ is defined in \eqref{F1}. 
Therefore, the LaSalle invariance principle holds and implies that bounded solutions converge to equilibria, and any $\omega$-limit set consists of a single equilibrium, in case of hyperbolic equilibria. See \cite[Section 4.3]{Henry81} and \cite{Matano88}.

The global attractor is thereby decomposed as $\mathcal{A}=\mathcal{E}\cup \mathcal{H}$, where $\mathcal{E}$ denotes the set of equilibria (time independent solutions, i.e., $u_t=0$) and $\mathcal{H}$ stands for the set of heteroclinic orbits, i.e., a solution $u(t)\in \mathcal{H}$ that satisfies
\begin{equation}\label{defHETS}
        u_-\xleftarrow{t\to -\infty}u(t)\xrightarrow{t\to \infty} u_+, \qquad \text{ for some } u_\pm\in\mathcal{E}.
\end{equation}
The task of explicitly finding equilibria and which heteroclinics occur is often called the \emph{connection problem}.
In particular, necessary and sufficient conditions are given in order to guarantee the occurrence of heteroclinics among two given equilibria, such as in \eqref{defHETS}.
The global attractors of scalar unidimensional parabolic equations \eqref{PDE} are known as \emph{Sturm attractors}, since the connection problem can be solved by means of nodal properties firstly discovered by Sturm \cite{Sturm} for the autonomous linear case, which was further rediscovered (and notably generalized) by Matano \cite{Matano82} for the non-autonomous semilinear case. %See also Angenent \cite{Angenent88}.
This characterization of Sturm attractors was carried out in the semilinear context with Hamiltonian reaction term $f(u)$ by Brunovský and Fiedler \cite{FiedlerBrunovsky89}, in the more general reaction term $f(x,u,u_x)$ by Fiedler and Rocha \cite{FiedlerRocha96}, and with periodic boundary conditions by Fiedler, Rocha and Wolfrum \cite{FRW,FRW12}. The author pursued quasilinear equations in \cite{Lappicy18} and singular diffusion in \cite{LappicySing}.
Some examples can be seen in \cite{Fiedler94}, and a classification of Sturm attractors of dimension 2 and 3 can be found in the respective triptychs \cite{FiRoPlanar1,FiRoPlanar2,FiRoPlanar3} and \cite{FiedlerRochatryp1,FiedlerRochatryp2,FiedlerRochatryp3}. 

Next, we establish the conjecture stated by Fiedler in \cite{Fiedler96} of constructing the Sturm attractors for fully nonlinear equations.

For the statement of the main theorem, we need a few definitions. 
An equilibrium $u_*(x)$ is \emph{hyperbolic} if the linearization of the right hand side of \eqref{PDE} at $u_*$ has no eigenvalue in the imaginary axis; also, the number of positive eigenvalues is called the \emph{Morse index}. % (i.e., the dimension of the unstable manifold of said equilibrium).
The \emph{zero number} $z(u_*)$ denotes the number of strict sign changes of a continuous function $u_*(x)$, which is also defined in \eqref{zeronumberdef} for time-dependent solutions. %An equilibrium $u_*$ is a solution of \eqref{PDE} when $u_t=0$.
Lastly, we say that two different equilibria $u_-,u_+$ of \eqref{PDE} are \emph{adjacent} if there does not exist an equilibrium $u_*$ of \eqref{PDE} such that $u_*(0)$ lies between $u_-(0)$ and $u_+(0)$, i.e. $u_-(0)<u_*(0)<u_+(0)$ or $u_-(0)>u_*(0)>u_+(0)$, and
\begin{equation}\label{adjintro}
    z(u_--u_*) = z(u_--u_+) = z(u_+-u_*).
\end{equation}
%This notion was described by Wolfrum \cite{Wolfrum02}.
%
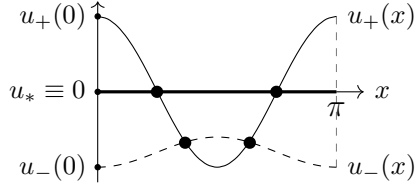
\begin{figure}[H]\centering
\begin{tikzpicture}[scale=1]
    \draw[->] (0,0) -- (3.5,0) node[right] {\footnotesize{$x$}};
    \draw[->] (0,-1.2) -- (0,1.2);
    \draw[color=gray,dashed] (3.14,1) -- (3.14,-1); \draw[color=gray,dashed] (3.14,0.01) -- (3.14,0) node[anchor=north ] {\color{black} $\pi$};

    %bdries
    \filldraw [black] (0,0) circle (1pt);
    \filldraw [black] (0,1) circle (1pt)node[left] {\footnotesize{$u_+(0)$}};
    \filldraw [black] (0,-1) circle (1pt)node[left] {\footnotesize{$u_-(0)$}};

    %curves    
    \draw [domain=0:3.14,variable=\t,smooth] plot ({\t},{cos(2*\t r)}) node[right] {\footnotesize{$u_+(x)$}};
    \draw [dashed,domain=0:3.14,variable=\t,smooth] plot ({\t},{-0.2*cos(2*\t r)-0.8}) node[right] {\footnotesize{$u_-(x)$}};
    \draw[very thick] (3.14,0) -- (0,0) node[left] {\footnotesize{$u_*\equiv 0$}};

    %intersections
    \emph{\filldraw [ultra thick] (0.78,0) circle (1.5pt);
        \filldraw [ultra thick] (2.35,0) circle (1.5pt);
        \filldraw [ultra thick] (1.15,-0.675) circle (1.5pt);
        \filldraw [ultra thick] (2,-0.675) circle (1.5pt);
        }
\end{tikzpicture}
\captionof{figure}{An example of continuous functions $u_\pm(x),u_*(x)$ such that $u_*$ is between $u_\pm$ at $x=0$, $z(u_--u_+) = z(u_+-u_*)=2$, but $z(u_--u_*)=0$. Note that the zero number of the difference of solutions counts the number of intersections of the two graphs. Thus, $u_-$ and $u_+$ are adjacent.}
\end{figure}
    
Both the zero number and Morse index can be computed from a permutation of the equilibria, called the \emph{Fusco-Rocha Permutation}, as it was done for the semilinear case in \cite{FuscoRocha,FiedlerRocha96}. 
We use the shooting method to explicitly find the equilibria and unravel the necessary information on adjacency by means of the Fusco-Rocha permutation for the fully nonlinear case in Section \ref{sec:perm}.
%Then, we discuss an explicit example \textcolor{red}{of a fully nonlinear PDE} that yields the Chafee-Infante attractor in Section \ref{sec:CIfully}. 
For such, it is required that the flow of the equilibria equation of \eqref{PDE} exists for all $x\in [0,\pi]$. 
\begin{thm}\emph{\textbf{Sturm Attractors for fully nonlinear PDEs.} } \label{attractorthmquasi}
Let $\alpha,\beta\in (0,1)$ such that $\alpha>1/2$.
Consider $f\in C^2$ satisfying the parabolicity condition \eqref{par} such that
$F$ and $\tilde{F}$ can be globally solved yielding the PDEs \eqref{diffIFTsplit} and \eqref{diffIFT2split}. Suppose that solutions of \eqref{diffIFT2split} generate a bounded and dissipative semiflow, which can be achieved with the sufficient conditions \eqref{diss}. Assume further that all equilibria for the equation \eqref{PDE} are hyperbolic. Then, 
\begin{enumerate}
    \item the global attractor $\mathcal{A}$ of \eqref{PDE} consists of finitely many equilibria $\mathcal{E}$, and heteroclinic connections $\mathcal{H}$ between them.
    \item there is a heteroclinic orbit $u(t)\in \mathcal{H}$ that converges to distinct equilibra $u_\pm\in\mathcal{E}$ as $t\to \pm\infty$, i.e.,
    \begin{equation}\label{hetmain}
        u_-\xleftarrow{t\to -\infty} u(t)\xrightarrow{t\to +\infty} u_{+}
    \end{equation}
    if, and only if, $u_-$ and $u_+$ are adjacent and $i(u_-)>i(u_+)$.
\end{enumerate}
\end{thm}

%This result meets the expectations of Fiedler \cite{Fiedler96}, that mentions fully nonlinear equations should yield the same type of attractors as the semilinear ones. 

The remaining of the paper is organized as follows. 
We firstly introduce the necessary background in Section \ref{sec:func}, including the appropriate functional setting, invariant stable/unstable manifolds, the dropping lemma and two of its consequences (a comparison of zero numbers and Morse indices within invariant manifolds, and the Morse-Smale property). In Section \ref{sec:globalsturm}, we build upon the background tools in order to establish the connection problem. 
In Section \ref{sec:perm}, we describe the shooting method that unravel the information on adjacency (i.e., Morse indices and zero numbers), which is encoded in a permutation of the equilibria. %The shooting is similar to the semilinear case. 
In Section \ref{sec:CIfully}, we provide an example of fully nonlinear equations yielding the well known Chafee-Infante attractor. Lastly, we discuss the present results and future directions in Section \ref{sec:disc}.

\section{Proof of main result}

\subsection{Background} \label{sec:func}

The phase-space $X^\alpha$ lies in the space of H\"older continuous functions $X:=C^\beta([0,\pi])$ with H\"older coefficient $\beta\in (0,1)$, %, intersected with the Neumann boundary conditions,
constructed as follows. See \cite{Lunardi95}, \cite{BabinVishik92}. The notation $C^{\beta}$ for some $\beta\in\mathbb{R}_+$ indicates that $\beta$ can be rewritten as $\lfloor\beta\rfloor+\{\beta\}$, where the integer part $\lfloor\beta\rfloor\in\mathbb{N}$ denotes the $\lfloor\beta\rfloor$-times differentiable functions whose $\lfloor\beta\rfloor$-derivatives is $\{\beta\}$-Hölder, where $\{\beta\}\in[0,1)$ is the fractional part of $\beta$.

Equation \eqref{diffIFT2} can be seen as an abstract differential equation on a Banach space,
\begin{equation}\label{linequiv}
    u_t=Au+g(u)    
\end{equation}
where $A:D(A)\rightarrow X$ is the linearization of the right-hand side of \eqref{diffIFT2} at the initial data $u_0$, and $g(u):=\tilde{F}(x,u,u_x,u_{xx})-Au$ is the Nemitskii operator taking values in $X$. The domain of $A$ is $D(A):=C^{2,\beta}([0,\pi])\cap \{\text{Neumann b.c.}\}\subseteq X$ for $\beta\in (0,1)$. Moreover, consider the interpolation spaces $X^\alpha:=C^{2\alpha+\beta}([0,\pi])$ between $X$ and $D(A)$, with $\alpha\in (0,1)$, and thus $A$ generates a strongly continuous semigroup in $X^\alpha$. 
Note that $X^\alpha\subseteq C^1([0,\pi])$ for $2\alpha+\beta>1$.
For a delicate analysis regarding the regularity theory for fully nonlinear parabolic equations, see \cite{Krylov, WangI, WangII, Crandall}, and more recently in \cite{CarafelliStefanelli,DongKrylov,ImbertSilvestre13}. % and references therein.
Therefore, solutions of the equation \eqref{PDE} define a local semiflow in $X^\alpha$ according to the variation of constants formula. Moreover, this semiflow is dissipative for $\alpha>1/2$ and due to the conditions \eqref{diss} according to Theorem 8.1.1 in Lunardi \cite{Lunardi95}. Note that the closure of orbits $\{ u(t) \text{ $ | $ } t\in\mathbb{R}_+\}$ is compact in $X^\alpha$, due to the compact embedding of Hölder spaces $X^{\tilde{\alpha}}\subseteq\subseteq X^\alpha$ for $0<\alpha<\tilde{\alpha}<1$ and Theorem 3.3.6 in \cite{Henry85}.
%http://matwbn.icm.edu.pl/ksiazki/sm/sm96/sm9639.pdf
%see the holder compact embedding in 24.14 below %http://www.math.ucsd.edu/~bdriver/231-02-03/Lecture_Notes/Holder-spaces.pdf
%
%Hence the equation \eqref{PDE} with the dissipative conditions \eqref{diss} defines a dissipative dynamical system in $X^\alpha$. 

In particular, this settles existence and uniqueness of solutions. For qualitative properties of solutions, such as the existence of invariant manifolds of a hyperbolic equilibrium $u_*\in\mathcal{E}$, one needs the spectrum of the linearization $A_*$ at $u_*$.
Note that $A_*u=\lambda u$ is a regular Sturm-Liouville problem, since the coefficients depend only on $x$ and are all bounded. Therefore, the spectrum of $A_*$ consists of non-zero real simple eigenvalues $\lambda_k$ accumulating at $-\infty$, and corresponding eigenfunctions $\phi_k(x)$ which form an orthonormal basis of $L^2$, and thereby of $X^\alpha\subseteq L^2$ by inheritance of the inner product. % of $L^2$. % as a linear subspace.
%that allows us to obtain the following filtration of invariant manifolds with respective linear asymptotic behavior. 
Thus we obtain the local behavior of \eqref{linequiv} nearby $u_*\in\mathcal{E}$, since there is a spectral gap between pairs of eigenvalues and $g(u)$ has small Lipschitz constant in a sufficiently small neighborhood of $u_*$ in which we can cut-off the nonlinearity outside such a neighborhood. 
\begin{prop} \label{hierarchyquasi}
    \emph{\textbf{Filtration of Invariant Manifolds.} \cite{Henry81, Angenent86, FiedlerBrunovsky86, Mielke91, BabinVishik92, Lunardi95}.}
    Let $u_*\in\mathcal{E}$ be a hyperbolic equilibrium of \eqref{PDE} with Morse index $n:=i(u_*)$. Then there exists a filtration of the \emph{unstable manifold}\footnote{Indices are not in agreement with the dimension of each submanifold, but with the number of zeros of the corresponding eigenfunction, e.g. $\phi_k$ has $k$ simple zeroes, whereas $\dim W^u_k=k+1$.} according to %$W^u_k\backslash W^u_{k-1}$ has, ie, $z(\phi_k)=k$.
    \begin{equation}
        W^u_0(u_*)\subset ... \subset W^u_{n-1}(u_*)=W^u(u_*)
    \end{equation}
    where each $W^u_k$ is invariant and has dimension $k+1$ with tangent space at $u_*$ spanned by $\phi_0,...,\phi_{k}$. 
    Analogously, there is a filtration of the \emph{stable manifold}
    \begin{equation}
        ... \subset W^s_{n+1}(u_*)\subset W^s_{n}(u_*)=W^s(u_*)
    \end{equation}
    where each $W^s_k$ is invariant and has codimension $k$ with tangent space spanned by $\phi_{k},\phi_{k+1},...$ . Moreover, the following linear asymptotic behavior holds true:
    \begin{enumerate}
    \item Let $u(t)\in W^u_k(u_*)\backslash W^u_{k-1}(u_*)$ with $k=0,...,i(u_*)-1$, where $W^u_{-1}(u_*):=\emptyset$. Then,
        \begin{equation}
        \frac{u(t)-u_*}{||u(t)-u_*||}\xrightarrow{t\rightarrow -\infty} \pm \phi_k \qquad \text{ in } C^1.
        \end{equation}
    \item Let $u(t)$ in $W^s_k(u_*)\backslash W^s_{k+1}(u_*)$ with $k\geq i(u_*)$. Then,
        \begin{equation}
        \frac{u(t)-u_*}{||u(t)-u_*||}\xrightarrow{t\rightarrow \infty} \pm \phi_k, \qquad \text{ in } C^1.
        \end{equation}
    \end{enumerate}
    %where the convergence takes place in $C^1$. 
\end{prop}
For the fully nonlinear case, the standard existence theorem of unstable manifolds only guarantees that $W^u(u_*)$ is locally diffeomorphic to a ball, see \cite[Theorem 9.1.4]{Lunardi95}. However, for the semilinear case, one can obtain global topological properties of unstable manifolds of equations \eqref{PDE}: the topological boundary of each $W^u(u_*)$ is not only homeomorphic to a sphere, but its interior is a ball, see \cite{FiedlerRocha15}. This excludes Alexander horned spheres and lens spaces.
Moreover, the unstable manifolds provide a cell decomposition of the global attractor $\mathcal{A}$ as a regular CW-complex, see \cite{FiedlerRochaCW14}. We believe these results still hold for the fully nonlinear case.

\begin{comment}
An important property is the behavior of solutions within each submanifold of the above filtration of the unstable or stable manifolds. This holds mainly due to the spectral gap of the linearized operator, and geometrically imply that solutions respect the filtration of the unstable manifold are tangent to its corresponding eigenfunction.
%
\begin{prop} \label{ConvEFquasi}
    \emph{\textbf{Linear Asymptotic Behavior.} \cite{Henry81, Angenent86, FiedlerBrunovsky86}.}
    %
    Consider a hyperbolic equilibrium $u_*\in\mathcal{E}$ with Morse index $n:=i(u_*)$ and a solution $u(t)$ of \eqref{PDE}. %The following holds,
    \begin{enumerate}
    \item \textcolor{red}{Let} $u(t)\in W^u_k(u_*)\backslash W^u_{k-1}(u_*)$ with $k=0,...,i(u_*)-1$, where $W^u_{-1}(u_*):=\emptyset$. Then,
        \begin{equation}
        \frac{u(t)-u_*}{||u(t)-u_*||}\xrightarrow{t\rightarrow -\infty} \pm \phi_k \qquad \text{ in } C^1.
        \end{equation}
    \item \textcolor{red}{Let} $u(t)$ in $W^s_k(u_*)\backslash W^s_{k+1}(u_*)$ with $k\geq i(u_*)$. Then,
        \begin{equation}
        \frac{u(t)-u_*}{||u(t)-u_*||}\xrightarrow{t\rightarrow \infty} \pm \phi_k, \qquad \text{ in } C^1.
        \end{equation}
    \end{enumerate}
    %where the convergence takes place in $C^1$. 
\end{prop}
%
\end{comment}

The conclusions of 1. and 2. regarding the asymptotic behavior in Proposition \ref{hierarchyquasi} also hold true by replacing the difference $u(t)-u_*$ with the tangent vector $u_t$.
Indeed, both $v:=u_t$ and $v:=u_1-u_2$, where $u_1,u_2$ are solutions of \eqref{diffIFT2}, satisfy a linear equation of the type
\begin{equation}\label{linPDE}
    v_t=a(t,x)v_{xx}+b(t,x)v_x+c(t,x)v
\end{equation}
where $x\in(0,\pi)$ has Neumann boundary conditions. The coefficients $a(t,x),b(t,x),c(t,x)$ are bounded for all $(t,x)\in \mathbb{R}\times [0,\pi]$ and given by
\begin{subequations}
\begin{align}
    {a}(t,x)&:=\int_1^2 \tilde{F}_q(x,u^s,u_x^s,u_{xx}^s )ds>0,\\
    {b}(t,x)&:=\int_1^2 \tilde{F}_p(x,u^s ,u_x^s,u_{xx}^s )ds,\\
    {c}(t,x)&:=\int_1^2 \tilde{F}_u(x,u^s ,u_x^s,u_{xx}^s )ds,   
\end{align}
\end{subequations}
where $q:=u_{xx},p:=u_x$, and $u^s :=(2-s) u_1 +(s-1) u_2 $ for $s\in[1,2]$. 
\begin{figure}[H]
\minipage{0.49\textwidth}\centering
   \begin{tikzpicture}[scale=1.8] 
    %manifold
    \draw [<<->>,very thick, domain=3.14:0,variable=\t,smooth] plot ({cos(\t r)},{sin(\t r)}) node[right] {\footnotesize{$W^u_0(u_*)$}};

    %equator
    \draw (-1,0) arc (180:360:1cm and 0.2cm);
    \draw[dashed] (-1,0) arc (180:0:1cm and 0.2cm);

    %greenwich
    \draw[rotate=-90,->] (-1,0) arc (180:281:1cm and 0.17cm);
    \draw[rotate=-90,dashed,->] (-1,0) arc (180:90:0.8cm and 0.17cm);
    
    %W_1 not in W_0
    \draw [domain=-0.05:1.6,variable=\t,smooth] plot ({0.3*\t-0.52},{0.415-0.58*(cos(2*\t r))}); %left
    \draw [domain=0.35:0.8,variable=\t,smooth] plot ({0.5*\t-0.2},{0.8+0.99*(cos(4*\t r))}); %right 

    \draw[->,rotate=-40] (-0.28,-0.45) -- (-0.28,-0.451);%left
    \draw[->,rotate=40] (0.027,-0.27) -- (0.027,-0.271);%right

    %equilibria
    \filldraw (-0.04,1) circle (1.4pt) node[anchor=south]{\footnotesize{$u_*$}};
    \end{tikzpicture}
\endminipage\hfill 
\minipage{0.49\textwidth}\centering
    \begin{tikzpicture}[scale=0.9]
    %coordinates
    \draw[very thick,<<->>] (-1.5,0) -- (1.5,0) node[right] {\footnotesize{$\phi_0$}};
    \draw[<->] (0,-1.4) -- (0,1.4) node[above] {\footnotesize{$\phi_1$}};    
    
    %equilibria
    \filldraw (0,0) circle (2.5pt) node[anchor=south west]{\footnotesize{$u_*$}};
    
    %tangencies
    \draw[domain=-0.97:0.97,samples=100,<->] plot ({0.03+\x^2*tan((\x)r)},\x);    
    \draw[domain=-0.97:0.97,samples=100,<->] plot ({-0.03-\x^2*tan((\x)r)},\x);    
    \end{tikzpicture}
\endminipage
\caption{A two-dimensional unstable manifold, $W^u(u_*)=W^u_1(u_*)$, with tangent space spanned by $\phi_0,\phi_1$. On the left, the unstable manifold contains the (bold) one-dimensional curve, i.e. $W^u_0(u_*)$. Solutions in $W^u_0(u_*)$ (resp. $W^u_1(u_*)\backslash W^u_{0}(u_*)$) are well-approximated by $\phi_0$ (resp. $\phi_1$) as $t\to -\infty$. 
On the right, the linear semiflow in the tangent space %of the unstable manifold of the equilibrium $u_*$ that 
that locally approximates the nonlinear semiflow. 
} \label{FIGunstable}
\end{figure}
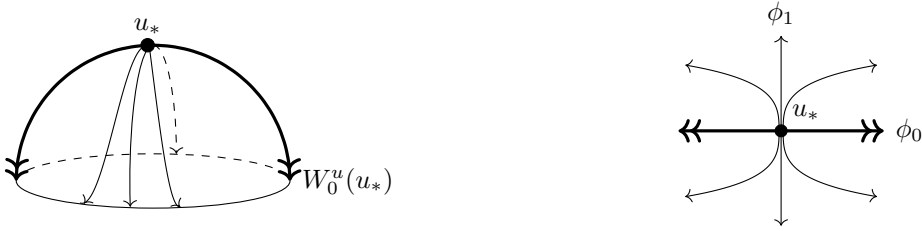

A fundamental ingredient in order to solve the connection problem are the nodal properties, i.e., the zero number of certain solutions of \eqref{PDE} is nonincreasing in time $t$, and decreases whenever a multiple zero occur.
A point $(t_*,x_*)\in\mathbb{R}\times [0,\pi]$ such that $u(t_*,x_*)=0$ is said to be %a \emph{simple zero} of $u(t,.)\in C^1$ if $u_x(t_*,x_*)\neq 0$ and 
a \emph{multiple zero} if $u_x(t_*,x_*)=0$.
%Different versions and proofs of this well known fact are due to Sturm \cite{Sturm}, Matano \cite{Matano82}, Angenent \cite{Angenent88} and others. %See \cite{Lappicy16drop} for a more recent account. 
Let the \emph{zero number} $0\leq z(u(t,.))\leq \infty$ count the number of strict sign changes in $x$ of a $C^1$ function $u(t,x)\not \equiv 0$, for each fixed $t$. More precisely, if $x\to u(t,x)$ is not of constant sign, let
\begin{equation}\label{zeronumberdef}
    z(u(t,.)):= \sup_k \left\{  
        \begin{array}{c} 
        \text{There is a partition $\{ x_j\}_{j=1}^{k}$ of } [0,\pi]\\
        \text{such that } u(t,x_j)u(t,x_{j+1})<0 \text{ for all } j=0,\cdots,k-1
        \end{array} \right\}.
\end{equation}
For functions which do not change sign, $x\mapsto u(t,x)\neq 0$, and we define $z(u):=0$. For the trivial constant, $x\mapsto u(t,x)\equiv 0$, we define $z(u\equiv 0):=-1$. 
\begin{lem} \label{droplemquasi}
    \emph{\textbf{Dropping Lemma.} \cite{Sturm,Matano82,Angenent88}.}
    Consider a non-trivial solution $v\in C^1$ of the linear equation \eqref{linPDE} for $t\in [0,T)$. Then, its zero number $z(v(t,.))$ satisfies
    \begin{enumerate}
    \item $z(v(t,.))<\infty$ for any $t\in (0,T)$.
    \item $z(v(t,.))$ is nonincreasing in time $t$.
    \item $z(v(t,.))$ decreases at multiple zeros $(t_*,x_*)$ of $v(t,.)$, i.e., 
    \begin{equation}
        z(v(t_*-\epsilon,.))>z(v(t_*+\epsilon,.)), \qquad \forall \text{ sufficiently small } \epsilon>0.
    \end{equation} 
    %for any sufficiently small $\epsilon>0$.
    \end{enumerate}
\end{lem}
Recall that both the tangent vector $v:=u_t$ and the difference $v:=u_1-u_2$ of solutions $u_1,u_2$ of the nonlinear equation \eqref{PDE} satisfy the linear equation \eqref{linPDE}, and thereby the proof of Lemma \ref{droplemquasi} for fully nonlinear equations is the same as in the semilinear case. %See \cite{Lappicy16drop}.
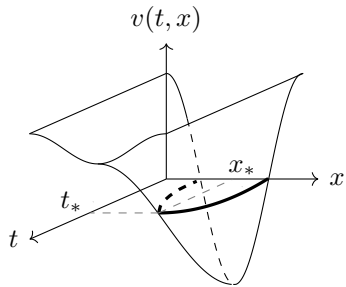
\begin{figure}[H]\centering
    \begin{tikzpicture}[scale=2]
    \draw[->] (0,0) -- (1,0) node[anchor=west]{\footnotesize{$x$}};
    \draw[->] (0,0) -- (0,0.9) node[anchor=south]{\footnotesize{$v(t,x)$}};    
    %\draw[->,rotate=114] (0,0) -- (0,1) node[anchor=east]{$t$};
    \draw[->] (0,0) -- (-0.9,-0.4) node[anchor=east]{\footnotesize{$t$}};

    %curve at t=0
    \draw [domain=0:0.1375,variable=\t,smooth] plot ({\t},{0.7*(cos(7*\t r))});    
    \draw [dashed, domain=0.1375:0.44,variable=\t,smooth] plot ({\t},{0.7*(cos(7*\t r))});    
    \draw [domain=0.44:0.9,variable=\t,smooth] plot ({\t},{0.7*(cos(7*\t r))});    
    
    %curve at t_*
    %\draw [dashed,color=gray,domain=0:0.9,variable=\t,smooth] plot ({-0.6+\t},{0.1+0.3*(cos(7*\t r))});
    
    %curve at t=1
    \draw [domain=0:0.9,variable=\t,smooth] plot ({-0.9+\t},{0.2+0.1*(cos(7*\t r))}); 
    %from t=0 to t=1
    \draw[-] (0,0.7) -- (-0.9,0.3);
    \draw[-] (0.9,0.7) -- (0,0.3);

    %minima
    \draw [domain=0:0.44,variable=\t,smooth] plot ({-0.45+2*\t},{-0.3+0.4*(cos(7*\t r))});    

    %zero curve
    \draw [very thick, domain=0:0.72,variable=\t,smooth] plot ({-0.05+\t},{-0.225+0.44*\t^2});    
    \draw [very thick, dashed, domain=0:0.25,variable=\t,smooth] plot ({-0.05+\t},{-0.225+0.42*\t^(1/2)});    

    %labels
    \draw[color=gray,dashed] (-0.05,-0.225) -- (-0.5,-0.225);\filldraw (-0.48,-0.15) circle (0.001pt) node[left] {\footnotesize{$t_*$}};
    \draw[color=gray,dashed] (-0.05,-0.225) -- (0.44,0);\filldraw (0.5,-0.05) circle (0.001pt) node[above] {\footnotesize{$x_*$}};    
    \end{tikzpicture}
    \caption{Example of a function $v(t,x)$ with zeros denoted in bold. The number of zeros (in $x$) of $v(t,x)$ is two when $t=0$, and this number decreases with time, since $v(t,x)$ has no zeros for $t>t_*$. Moreover, the function $v(t_*,x)$ has a multiple zero at $x_*$, which is when the dropping occurs.}
\end{figure}
We can now combine the dropping lemma \ref{droplemquasi} and the asymptotic description in Proposition \ref{hierarchyquasi} in order to relate the zero number within invariant manifolds and the Morse indices of equilibria. 

\begin{prop} \label{Znuminvmfld}
    \emph{\textbf{Zero number within Invariant Manifolds.} \cite{FiedlerBrunovsky86}.}
    Consider a hyperbolic equilibrium $u_*\in\mathcal{E}$ with Morse index $i(u_*)$ and a solution $u(t)$ of \eqref{PDE}.
    \begin{enumerate}
    \item Let $u(t) \in W^u(u_*)$, then $i(u_*)>z(u(t)-u_*)$.
    \item Let $u(t) \in W^s_{loc}(u_*)\backslash \{u_*\}$, then $z(u(t)-u_*)\geq i(u_*)$.
    \end{enumerate}
    These results also hold by replacing $u(t)-u_*$ with the tangent vector $u_t$.
\end{prop}
%
%The above theorem implies that \eqref{PDE} has no homoclinic orbits. Indeed, if there were any, then $i(u_*)<i(u_*)$, which is a contradiction.
%
%Another imediate consequence of the above theorem, as stated in Brunovsky and Fiedler, is called \emph{Morse blocking}: some conditions on the zero number prevent connections. Namely, If there exists a heteroclinic connecting $v$ and $w$, then
%\begin{equation}\label{MorseBlock}
%\geq z\leq 
%\end{equation}
%\begin{thm} \label{MorseSmalequasi}
%    \emph{\textbf{Morse-Smale Property} \cite{Henry85}, \cite{Angenent86}, \cite{FuscoRocha} }
    
%    Consider two hyperbolic equilibria $u_-$ and $u_+$ with respective Morse indices $i(u_-),i(u_+)$. If $W^u(u_-)\cap W^s(u_+)\neq \emptyset$, then such intersection is transverse. Moreover, $W^u(u_-)\cap W^s(u_+)$ is an embedded submanifold of dimension $i(u_-)-i(u_+)$.
%\end{thm}

As a consequence of Proposition \ref{Znuminvmfld}, in case of hyperbolic equilibria, the semiflow of \eqref{PDE} is \emph{Morse-Smale}%in the sense of \cite{HaleMagalhaesOliva84}
, i.e., the non-wandering set %\footnote{\textcolor{red}{This set consists of points $u_0\in X^\alpha$ such that for any $\epsilon,T>0$, there exists $t>T$ such that the evolution $u(t)\cdot B_\epsilon\cap B_\epsilon \neq \emptyset$, where $B_\epsilon\subseteq X^\alpha$ denotes the ball of radius $\epsilon>0$ in $X^\alpha$ and $u(t)\cdot B_\epsilon$ denotes the evolution of such ball under the semiflow $u(t)$.}} 
consists of finitely many hyperbolic equilibria, and the stable and unstable manifolds of equilibria intersect transversely. %The former is a consequence of the Lyapunov function \eqref{BOOK} and the fact that hyperbolic equilibria forms a finite set, whereas the latter is displayed next.
\begin{thm} \label{MorseSmale}
    \emph{\textbf{Transversality.} \cite{Henry85,Angenent86,FuscoRocha}.}
    Consider hyperbolic equilibria $u_\pm \in\mathcal{E}$ with Morse indices $i(u_\pm)$. If $W^u(u_-)\cap W^s(u_+)\neq \emptyset$, then such intersection is transverse %, i.e., $X=T_{u_0}W^u(u_-)\oplus T_{u_0}W^s(u_+)$. 
    and $W^u(u_-)\cap W^s(u_+)$ is an embedded submanifold of dimension $i(u_-)-i(u_+)$.
\end{thm}
\begin{figure}[H]\centering
\begin{tikzpicture}[scale=1.9]
    %equilibria
    \filldraw [black] (0,-1) circle (1pt);%  node[anchor=north east] {\footnotesize{$u_-$}};
    \filldraw [black] (3.14,-1) circle (1pt); % node[anchor=north west] {\footnotesize{$u_+$}};
    
    %intersection
    \filldraw [black] (1.57,-0.6) circle (1pt);    
    
    %unstable mfld
    \draw [->,domain=0:0.3,variable=\t,smooth] plot ({\t},{-0.2*cos(2*\t r)-0.8});
    \draw [->,domain=0:-0.3,variable=\t,smooth] plot ({\t},{-0.2*cos(2*\t r)-0.8});
    
    \draw[<->] (0,-1.3) -- (0,-0.7) node[above] {\footnotesize{$W^u(u_-)$}};
    \draw (1.57,-0.9) -- (1.57,-0.3);

    \draw [domain=-0.3:1.57,variable=\t,smooth] plot ({\t},{-0.2*cos(2*\t r)-0.8-0.3});
    \draw [domain=-0.3:1.57,variable=\t,smooth] plot ({\t},{-0.2*cos(2*\t r)-0.8+0.3});    
    
    %stable mfld
    %\draw [dashed,->,domain=2.87:2.94,variable=\t,smooth] plot ({\t},{-0.2*cos(2*\t r)-0.8});%setinha 
    \draw [dashed,->,domain=3:3.06,variable=\t,smooth] plot ({\t},{-0.2*cos(2*\t r)-0.8});%setinha 

    \draw[dashed,shift={(1.58,-0.6)},rotate=114] (0,-0.26) -- (0,0.26); %linha transversal na intersecal

    \draw [dashed,domain=1.66:3.14,variable=\t,smooth] plot ({\t+0.2},{-0.2*cos(2*\t r)-0.8-0.1+0.2})   node[anchor= south] {\footnotesize{$W^s(u_+)$}}; %parte de cima

    \draw [dashed,domain=3.14:1.66,variable=\t,smooth] plot ({\t-0.3},{-0.2*cos(2*\t r)-0.8-0.1}); %parte debaixo
    \draw [dashed,domain=3.44:3.09,variable=\t,smooth] plot ({\t-0.2},{0.2*cos(2*\t r)-0.8-0.5});  %parte debaixo prolongada
    
    \draw [->,dashed,domain=3.35:3.19,variable=\t,smooth] plot ({\t},{0.2*cos(2*\t r)-0.8-0.4});  %heteroclinica prolongada
    
    \draw[->] (3.05,-1.1) -- (3.1,-1.04);%setinha transversal baixo
    \draw[<-,shift={(0.11,0.15)}] (3.05,-1.1) -- (3.1,-1.04);%setinha transversal cima
    
    %heteroclinic
    \draw [domain=0:1.57,variable=\t,smooth] plot ({\t},{-0.2*cos(2*\t r)-0.8});
    \draw [dashed,domain=1.57:3.14,variable=\t,smooth] plot ({\t},{-0.2*cos(2*\t r)-0.8});    

\end{tikzpicture}
\captionof{figure}{An example of a transverse heteroclinic orbit connecting two hyperbolic equilibria: from $u_-$ to $u_+$ as $t\in\mathbb{R}$ increases. The heteroclinic occurs as the intersection of $W^u(u_-)$ and $W^s(u_+)$.}
\end{figure}
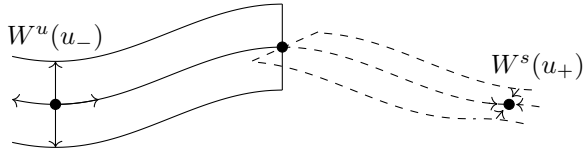
In particular, Morse-Smale systems satisfy the following transitivity principle: if there is a heteroclinic orbit from $u_-$ to $u_*$ and another from $u_*$ to $u_+$, then there is also a heteroclinic from $u_-$ to $u_+$. %See \cite{Angenent86} for a general proof of Theorem \ref{MorseSmale}.
This will be the main feature of Morse-Smale systems which will be used in the upcoming section.
%The main consequence of transversality which will be used in the upcoming section 

Note that the Morse-Smale property remains true in certain cases the equilibria are not hyperbolic, as in \cite{Henry85}; or higher spatial dimension, generically, see \cite{BrunovskyPolacik97}. 
We emphasize that genericity of both hyperbolicity and the Morse-Smale property have been proved for scalar unidimensional semilinear equations, which respectively imply the local and global stability of the dynamics with respect to perturbations of the system. See \cite{BrunovskyChow84,HaleMagalhaesOliva84,Lu94,Oliva02,BrunovskyJolyRaugel19}. These genericity results should remain true for fully nonlinear equations. %For the usage of a new notion of genericity for PDEs, called \emph{prevalence}, see X.

\subsection{Sturm global structure}\label{sec:globalsturm}

This section %gathers all tools from previous sections in order to 
abstractly constructs the attractor for the equation \eqref{PDE} and proves the second part of Theorem \ref{attractorthmquasi}.
Its proof is a consequence of four propositions. First, the \emph{cascading principle} guarantees it is enough to construct only heteroclinics between equilibria with Morse index differing by 1. Second, on one direction, the \emph{blocking principle}: certain conditions prevent the existence of a heteroclinic; whereas on the other direction, the \emph{liberalism principle}: if heteroclinics are not forbidden, then they actually exist.
%The cascading, blocking and liberalism principles assert that the information of the Morse indices $i(.)$ and zero numbers $z(.-.)$ are sufficient to construct the global attractor explicitly. These two quantities are determined by the shooting curves, and consequently the Fusco-Rocha permutation, as in Section \ref{sec:perm}. 
Lastly, \emph{Wolfrum's equivalence} yield a relation between two notions of adjacencies: one that depends on a cascade between equilibria, and one that does not.

\begin{prop}\emph{\textbf{Cascading Principle.} \cite{FiedlerRocha96}}. \label{cascading}
Consider two equilibria $u_\pm$ of equation \eqref{PDE} such that $n := i (u_-) - i(u_+) > 0$. Then the following statements are equivalent:
\begin{itemize}
\item[(i)] There exists a heteroclinic orbit from $u_-$ to $u_+$ in forward time, as in \eqref{hetmain}.
\item[(ii)] There exists a sequence (cascade) of equilibria $\{ e_j\}_{j=0}^n$ with $e_n:=u_-$ and $e_0:=u_+$ such that $i(e_{j+1})=i(e_j)+1$ and there exists a heteroclinic orbit from $e_{j+1}$ to $e_j$ in forward time for each $j=0,...,n-1$.
\end{itemize}
\end{prop}
The proof of Proposition \ref{cascading} relies on nodal properties, and we refer to \cite[Lema 1.5]{FiedlerRocha96}. The implication $(ii) \rightarrow (i)$ is a special case of a transitivity principle that holds for Morse-Smale systems, due to Theorem \ref{MorseSmale}, whereas the converse implication $(i) \rightarrow (ii)$ is not true in general and is specific of equation \eqref{PDE}.
Due to the \emph{cascading principle}, it suffices to construct all heteroclinic orbits between equilibria with Morse indices differing by one.
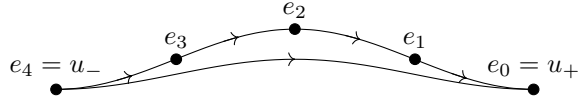
\begin{figure}[H]\centering
\begin{tikzpicture}[scale=2]
    %equilibria
    \filldraw [black] (0,-1) circle (1pt) node[anchor=south] {\footnotesize{$e_4=u_-$}};
    \filldraw [black] (3.14,-1) circle (1pt) node[anchor=south] {\footnotesize{$e_0=u_+$}};
    
    %cascade of equilibria
    \filldraw [black] (1.57,-0.6) circle (1pt) node[anchor=south] {\footnotesize{$e_2$}};    
    \filldraw [black] (0.79,-0.8) circle (1pt) node[anchor=south] {\footnotesize{$e_3$}};    
    \filldraw [black] (2.36,-0.8) circle (1pt) node[anchor=south] {\footnotesize{$e_1$}};        
    
    %heteroclinic
    \draw [domain=0:3.14,variable=\t,smooth] plot ({\t},{-0.2*cos(2*\t r)-0.8});
    
    %unstable mfld arrow
    \draw [->,domain=0.4:0.5,variable=\t,smooth] plot ({\t},{-0.2*cos(2*\t r)-0.8});
    \draw [->,domain=1.1:1.2,variable=\t,smooth] plot ({\t},{-0.2*cos(2*\t r)-0.8});
    \draw [->,domain=1.9:2,variable=\t,smooth] plot ({\t},{-0.2*cos(2*\t r)-0.8});    
    \draw [->,domain=2.6:2.7,variable=\t,smooth] plot ({\t},{-0.2*cos(2*\t r)-0.8});

    %direct heteroclinic
    \draw [domain=0:3.14,variable=\t,smooth] plot ({\t},{-0.1*cos(2*\t r)-0.9});
    \draw [->,domain=1.56:1.57,variable=\t,smooth] plot ({\t},{-0.1*cos(2*\t r)-0.9});    
\end{tikzpicture}
\captionof{figure}{A heteroclinic from $u_-$ to $u_+$ exists if, and only if there is a cascade of consecutive heteroclinics among equilibria $\{e_j\}_{j=0}^4$ with Morse index differing by 1 such that $e_4=u_-$ and $e_0=u_+$.
}
\end{figure}
The second result provides a condition that prevents heteroclinic orbits between equilibria with Morse indices differing by one. Before we present its content, we say that two hyperbolic equilibria $e_{j+1}$ and $e_j$ of \eqref{PDE} with $i(e_{j+1})=i(e_j)+1$ are \emph{blocked} if one of the following conditions holds:
 \begin{enumerate}
     \item \emph{Morse blocking:} $z(e_{j+1}-e_j)\neq i(e_j)$;
     \item \emph{Zero number blocking:} there exists an equilibria $u_*$ between $e_{j+1}$ and $e_j$ at $x=0$, i.e. $e_j(0)<u_*(0)<e_{j+1}(0)$ or $e_{j}(0)>u_*(0)>e_{j+1}(0)$, such that 
     \begin{equation}\label{zeronumberblock}
         z(e_{j+1}-u_*)=z(e_{j+1}-e_j)=z(e_j-u_*).
     \end{equation}
\end{enumerate} 
The proof of the upcoming Proposition \ref{estabhets} follows from nodal properties of solutions of \eqref{PDE}; see the subsequent discussion of Definition 1.6 in \cite{FiedlerRocha96}.
\begin{prop} \emph{\textbf{Blocking Principle.} \cite{FiedlerRocha96}.} \label{estabhets}
Consider two equilibria $e_{j+1}$ and $e_j$ of \eqref{PDE} such that $i(e_{j+1})=i(e_j)+1$. If $e_{j+1}$ and $e_j$ are blocked, then there does not exist heteroclinic orbits from $e_{j+1}$ to $e_j$ in forward time.
\end{prop}
The next claim is an act of liberalism: If a heteroclinic connection among two equilibria is not forbidden by the blocking law, then a connection between them exists. 
\begin{prop} \emph{\textbf{Liberalism Principle.}} \label{estabhets2}
Consider hyperbolic equilibria $e_{j+1},e_j\in\mathcal{E}$ of \eqref{PDE} such that $i(e_{j+1})=i(e_j)+1$. If $e_{j+1}$ and $e_j$ are not blocked, then there exists a heteroclinic orbit from $e_{j+1}$ to $e_j$ in forward time.
\end{prop}
%
%The cascading (Lemma 1.5 in \cite{FiedlerRocha96}) and blocking principles follow from the dropping lemma and its consequences, as Fiedler and Rocha \cite{FiedlerRocha96}. In particular, cascading arises from the Morse-Smale property of the system, and the usage of the dropping lemma and a continuous Lyapunov function - both ingredients exist for the fully nonlinear case; and hence the proof follows verbatim. There is only a mild modification in the proof of the liberalism principle (Lemma 1.7 in \cite{FiedlerRocha96}).
%
The proof of \emph{liberalism} in Proposition \ref{estabhets2} follows from an application of the Conley index theory, which can be applied due to pre-compactness of orbits; see \cite[Lemma 1.7 and Section 4]{FiedlerRocha96}.
We now provide the necessary modification of their proof. 

The Conley index can be applied to detect heteroclinics as follows. Construct a closed neighborhood $N$ such that its maximal invariant subspace is the closure of the set of heteroclinics between $u_\pm$, i.e. $\Sigma=\{u_\pm\}\cup(\overline{W^u(u_-)\cap W^s(u_+)})$.
Suppose, towards a contradiction, that there are no heteroclinics connecting $u_-$ and $u_+$, i.e., $\Sigma=\{ u_-,u_+\}$. Then, the index is given by the wedge sum $C(\Sigma)=[\mathbb{S}^n]\vee [\mathbb{S}^m]$, where $n>m$ are the respective Morse index of $u_-$ and $u_+$.
However, if one can prove that $C(\Sigma)=[0]$, %, where $[0]$ means that the index is given by the homotopy equivalent class of a point. 
this would yield a contradiction and there must be a connection between $u_-$ and $u_+$. Note that the Morse-Smale structure excludes connection from $u_+$ to $u_-$, and hence there is a connection from $u_-$ to $u_+$.
The two ingredients missing in the proof are the construction of an isolating neighborhood $N$ of $\Sigma$ and the proof that $C(\Sigma)=[0]$.

The isolating neighborhood of $\Sigma$ is given by $N_\epsilon (u_\pm):=B_\epsilon(u_-)\cup B_\epsilon(u_+) \cup K({u_\pm})$, where  $B_\epsilon(u_\pm)\subseteq X^\alpha$ are $\epsilon$-balls centered at $u_\pm$ and $K({u_\pm})$ is the closed set defined as
\begin{equation}
    K({u_\pm}):= \left\{ u \in X \mid 
        \begin{array}{c} 
        z(u-u_-)=i(u_+)=z(u-u_+) \\
        u_+(0)\leq u(0)\leq u_-(0) 
        \end{array} \right\}.
\end{equation}
Note that $N_\epsilon (u_\pm)$ has no equilibria besides $u_-$ and $u_+$ for sufficiently small $\epsilon>0$, due to hyperbolicity and the zero number blocking. % condition implies there are no equilibria in $K({u_\pm})$ besides possibly $u_-$ and $u_+$.

Moreover, the maximal invariant subset $\Sigma$ of $N_\epsilon$ is the set of the heteroclinics from $u_-$ to $u_+$ given by $\overline{W^u(u_-)\cap W^s(u_-)}$. 
On one hand, since $\Sigma$ is globally invariant, then it is contained in the attractor $\mathcal{A}$, which consists of equilibria and heteroclinics. Since there are no other equilibria in $N_\epsilon (u_\pm)$ besides $u_\pm$, then the only heteroclinics that can occur are between them.
On the other hand, Theorem \ref{Znuminvmfld} implies that along a heteroclinic $u(t)\in\mathcal{H}$ the zero number satisfies $z^t(u-u_\pm)=i(u_+)$ for all time, since $i(u_-)=i(u_+)+1$. Therefore $u(t)\in K({u_\pm})$ and the closure of the orbit is contained in $N_\epsilon (u_\pm)$. Since the closure of the heteroclinic is invariant, it must be contained in $\Sigma$.

Lastly, it is proven that $C(\Sigma)=[0]$ in three steps, yielding the desired contradiction. % and the proof of the theorem. %We modify the first and second step from \cite{FiedlerRocha96}, whereas the third remain the same. 
In the first step, a model is constructed displaying a saddle-node bifurcation with respect to a parameter $\mu\in \mathbb{R}$. For $n:=z(u_+-u_-)\in\mathbb{N}$ fixed,
\begin{equation}\label{prototype2}
    v_t=v_{\xi\xi}+\lambda_nv+ g_n(\mu,\xi,v,v_\xi)
\end{equation}
where $\xi\in [0,\pi]$ has Neumann boundary conditions, $\lambda_n=-n^2$ are the eigenvalues of the laplacian with respective eigenfunctions $\cos(n\xi)$, and 
\begin{equation}
    g_n(\mu,\xi,v,v_\xi):=\left(v^2+\frac{1}{n^2}v^2_\xi-\mu\right)\cos(n\xi).
\end{equation}
For $\mu>0$, we obtain that $v_\pm=\pm \sqrt{\mu}\cos(n\xi)$ are equilibria solutions of \eqref{prototype2} such that $z(v_+-v_-)=n$, since the $n$ intersections of $v_-$ and $v_+$ will be at its $n$ zeroes. 
Moreover, those equilibria are hyperbolic for small $\mu>0$, such that $i(v_+)=n+1$ and $i(v_-)=n$, see \cite{FiedlerRocha96}.

Note that the equilibria $v_\pm$ together with their connecting orbits form an isolated set $\Sigma_\mu(v_\pm):= \overline{W^u(v_-)\cap W^s(v_+)}$ with isolating neighborhood $N_\epsilon(v_\pm)$. Since the bifurcation parameter $\mu$ can be seen as a homotopy parameter, the Conley index is the one of a point by homotopy invariance, i.e., $C(\Sigma_\mu(v_\pm))=C(\Sigma_0(v_\pm))=[0]$.

In the second step, $v_-$ and $v_+$ are transformed respectively into $u_-$ and $u_+$.
Recall $n=z(v_--v_+)=z(u_+-u_-)$. Hence, choose $\xi(x)$ a smooth diffeomorphism of $[0,\pi]$ that maps the zeros of $v_-(\xi)-v_+(\xi)$ to the zeros of $u_-(x)-u_+(x)$. Therefore, the zeros of $v_-(\xi(x))-v_+(\xi(x))$ and $u_-(x)-u_+(x)$ occur in the same points in the variable $x\in [0,\pi]$.
Next consider the transformation 
\begin{align*}
    L: X &\to X   \\ 
    v(\xi)&\mapsto l(x)[v(\xi(x))-v_-(\xi(x))]+u_-(x),
\end{align*}
where $l(x)$ is defined pointwise through
%\begin{align*}
%    \xi:[0,\pi]&\to [0,\pi]\\
%    L &\mapsto \xi(L)
%\end{align*}
%and
\begin{align*}
    l(x):=
    \begin{cases}
        \frac{u_+(x)-u_-(x)}{v_+(\xi(x))-v_-(\xi(x))} &, \text{ if } v_+(\xi(x))\neq v_-(\xi(x)), \\
        \frac{\partial_x(u_{+}(x)-u_{-}(x))}{\partial_x(v_{+}(\xi(x))-v_{-}(\xi(x)))} &, \text{ if } v_+(\xi(x))= v_-(\xi(x)),
    \end{cases}
\end{align*}
such that the coefficient $l(x)$ is smooth and nonzero due to the l'H\^opital rule. Hence, $L(v_-)=u_-$ and $L(v_+)=u_+$ as desired. Note %we supposed $2\alpha+\beta>1$ so that solutions $u_\pm\in C^1$, hence $L$ is of this regularity as well. Moreover, $L$ is invertible with inverse having the same regularity. In particular, it 
that $L$ is a homeomorphism, and hence a homotopy equivalence.

Moreover, the number of intersections of functions is invariant under the map $L$,
%\begin{equation}
%    z(L (v(\xi)-\tilde{v}(\xi)))=z(l(x)[v(\xi(x))-\tilde{v}(\xi(x))])=z(v(x)-\tilde{v}(x))
%\end{equation}
and hence $K({v_\pm})$ is mapped to $K({u_\pm})$ under $L$. 
Consider $w(t,x):=L(v(t,\xi))$, hence the map $L$ transforms the equation \eqref{prototype2} into 
\begin{equation}\label{IDK}
    w_t=\tilde{a}(x)w_{xx}+\tilde{f}(x,w,w_x)
\end{equation}
where the Neumann boundary conditions are preserved, and the terms $\tilde{a},\tilde{f}$ can be obtained from $g_n$ and the transformations
$w=L(v)$ and $\xi(x)$, see \cite{FiedlerRocha96}.

Note that $v_\pm$ are mapped into $w_\pm:=L(v_\pm)=u_\pm$, which are equilibria of \eqref{IDK}, with same zero numbers and Morse indices as $v_\pm$ and $u_\pm$.
The isolated invariant set $\Sigma_\mu(v_\pm)$ is transformed into $L(\Sigma_\mu(v_\pm))=\Sigma_\mu(w_\pm)$, which is still isolated and invariant, with invariant neighborhood $L(N_\epsilon(v_\pm))=N_\epsilon(w_\pm)$. Moreover, the Conley index is preserved, since $L$ is a homotopy equivalence: $C(\Sigma_\mu(v_\pm))= C(L(\Sigma_\mu(v_\pm)))=C(\Sigma_\mu(w_\pm))$.

In the third step, we homotope the equation from $w$ to $u$, i.e., we homotope the diffusion coefficient $\tilde{a}$ and nonlinearity $\tilde{f}$ from the model \eqref{IDK} to the desired diffusion $\tilde{F}^1$ and reaction $\tilde{F}^0$ in the equation \eqref{diffIFT2split}. Indeed, consider the parabolic equation 
\begin{equation}\label{homotopyHET}
    u_t=\tilde{F}^\tau(x,u,u_x,u_{xx}),
\end{equation}
where
\begin{equation}
    \tilde{F}^\tau:=(1-\tau) (\tilde{a}u_{xx}+\tilde{f})+ \tau(\tilde{F}^1u_{xx}+\tilde{F}^0)+\sum_{j=- \text{ , } +}\chi_{u_j}\mu_{u_j}(\tau)[u-{u_j}(x)],
\end{equation}
and $\chi_{u_j}$ are cut-offs being 1 nearby $u_j$ and zero far from $u_j$, the coefficients $\mu_j(\tau)$ are zero near $\tau=0,1$ and shift the spectra of the linearization at $u_\pm$ appropriately such that uniform hyperbolicity of these equilibria is guaranteed during the homotopy. 

Therefore the equilibria $u_\pm$ are preserved throughout the homotopies. Moreover, their connecting orbits $\Sigma^\tau(u_\pm):=\overline{W^u(u_-)\cap W^u(u_+)}\subseteq K({u_\pm})$, for all $\tau\in [0,1]$, since the dropping lemma holds throughout the homotopy. The equilibria $u_\pm$ do not bifurcate as $\tau$ changes, due to hyperbolicity. Choosing $\epsilon>0$ small enough, the neighborhoods $N_\epsilon(u_\pm)$ form an isolating neighborhood of $\Sigma^\tau(u_\pm)$ throughout the homotopy. Indeed, $\Sigma^\tau(u_\pm)$ can never touch the boundary of $K({u_\pm})$, except at the points $u_\pm$ by the dropping lemma. Once again the Conley index is preserved by homotopy invariance, and thus $C(\Sigma(u_\pm))=C(\Sigma_\mu(w_\pm))=C(\Sigma_\mu(v_\pm))=[0]$.
This concludes the liberalism proof.

%Note we preferred the splitting \eqref{diffIFT2}, because the right-hand side of the homotopy in \eqref{homotopyHET} depends only on $x,u,u_x,u_{xx}$. If we chose the other splitting \eqref{diffIFTsplit}, the right-hand side of \eqref{homotopyHET} would depend on $u_t$ instead of $u_{xx}$. %We believe the former is more natural in this setting.

Propositions \ref{cascading}, \ref{estabhets} and \ref{estabhets2} yield the existence of heteroclinics between $u_-$ and $u_+$ such that $n := i (u_-) - i(u_+) > 0$ if, and only if they are \emph{cascadly adjacent}, i.e., in case there is a cascade of equilibria $\{ e_j\}_{j=0}^n$ with $e_0:=u_-$ and $e_n:=u_+$ such that for all $j=0,...,n-1$ the following holds:
  \begin{enumerate}
     \item $i(e_{j+1})=i(e_j)+1$,
     \item $z(e_j-e_{j+1})=i(e_{j+1})$,
     \item there are no equilibria $u_*$ between $e_{j}$ and $e_{j+1}$ at $x=0$ satisfying \eqref{zeronumberblock}.
 \end{enumerate} 
However, Theorem \ref{attractorthmquasi} yields existence of heteroclinics relying on the notion of adjacency in \eqref{adjintro}, which does not involve a cascade. These notions of adjacency coincide, and this is the core of Wolfrum's ideas in \cite{Wolfrum02}. See also \cite[Appendix 7]{FiedlerRochatryp2}.

\begin{prop} \emph{\textbf{Wolfrum's equivalence.}} \label{wolfrumlemma}
Consider equilibria $u_\pm\in\mathcal{E}$ such that $i(u_-)>i(u_+)$. The equilibria $u_\pm$ are adjacent if, and only if they are cascadly adjacent.
\end{prop}

\subsection{Shooting: Finding Equilibria and Computing Adjacency}\label{sec:perm}

The next step on our quest to construct Sturm attractors is to find all equilibria of \eqref{PDE} and compute their Morse indices and zero numbers, in order to discern which equilibria are adjacent. We will use shooting methods similar to \cite{FuscoHale85,RochaHale85,Rocha85,Rocha88,Rocha94}. %We mention that a permutation associated to the equilibria is enough to compute adjacency, firstly constructed by Fusco and Rocha \cite{FuscoRocha}.

The equilibria equation associated to \eqref{PDE} is
\begin{equation}\label{eqeq}
    0=f(x,u,u_x,u_{xx},0) 
\end{equation}
for $x \in [0,\pi]$ with Neumann boundary conditions.
Equation \eqref{eqeq} can be translated as follows, according to the implicit function theorem, as in \eqref{diffIFT} and \eqref{diffIFTsplit},
\begin{equation}\label{IFTshoot}
    u_{xx}=F^0(x,u,u_x).
\end{equation}
In turn, equation \eqref{IFTshoot} can be reduced to a first order autonomous system,
\begin{subequations}\label{shootflow}
\begin{align}
    u'&= p,\\
    p'&=F^0(x,u,p),\\
    x'&=1
\end{align}
\end{subequations}
where the Neumann boundary condition amounts to 
$p=0$ at $x=0,\pi$ and $(.)'$ denotes a derivative with respect to a new parameter $\tau:=x\in [0,1]$. We suppose solutions of \eqref{shootflow} exist for all $x\in[0,\pi]$ and any initial data with $(x,u,p)=(0,u_0(0),0)$.

Equilibria of the PDE \eqref{PDE} can be found as follows. They must be in the line 
\begin{equation}
    L_0:=\{(x,u,p)\in\mathbb{R}^3 \textbf{ $|$ }  (x,u,p)=(0,a,0) \text{ and } a\in\mathbb{R}\},  
\end{equation}
because of Neumann boundary at $x=0$. Then, this line can be evolved under the flow of \eqref{shootflow} until $x=\pi$, yielding the \emph{shooting manifold} defined as
\begin{equation}
    M:= \{(x,u,p)\in[0,\pi]\times \mathbb{R}^3 \textbf{ $|$ } (x,u(x,a,0),p(x,a,0))  \text{ and } a\in\mathbb{R}  \},
\end{equation}
where $(x,u(x,a,0),p(x,a,0))$ is the solution of \eqref{shootflow} which evolves the initial data $(0,a,0)$.
Denote by $M_x$ the cross-section of $M$ for some fixed $x\in [0,\pi]$, which is a curve parametrized by $a\in\mathbb{R}$. %From now on we supress the upper indices from the graph functions describing each manifold, but it is still clear which graph we're talking about by context.
We call the cross-section $M_\pi$ by \emph{shooting curve}, which is an object that carries all the information regarding existence, hyperbolicity and adjacency of equilibria, as in the following Lemmata. See \cite{Rocha85,Hale99,Rocha91,FiedlerRocha21}. Note that intersections of the curve $M_\pi$ and the line $L_\pi:=\{(x,u,p)\in\mathbb{R}^3 \textbf{ $|$ }  (x,u,p)=(0,b,\pi) \text{ and } b\in\mathbb{R}\}$ also satisfy Neumann boundary conditions at $x=\pi$, and consequently yield equilibria solutions of the PDE \eqref{PDE}.
%
%All we have to guarantee is that equation \eqref{shootflow} yields the existence of equilibria, their Morse indices and zero numbers through a permutation.
%
\begin{figure}[H]\centering
    \begin{tikzpicture}[scale=0.5]
    \draw[->] (-4,0) -- (4,0) node[anchor=west]{\footnotesize{$u$}};
    \draw[->] (0,0) -- (0,2) node[anchor=south]{\footnotesize{$p$}};    
    %\draw[dashed] (0,-2) -- (0,0);    
    \draw[->,rotate=116.5] (0,0) -- (0,6) node[anchor=east]{\footnotesize{$x$}};

    %stressing equilibria
    %\draw[shift={(-1,0)},rotate=114] (0,0) -- (0,1);
    %\draw[shift={(1,0)},rotate=114] (0,0) -- (0,1);
    \draw[dashed,shift={(-3,0)},rotate=116.5] (0,1.1) -- (0,4.5);
    \draw[shift={(-3,0)},rotate=116.5] (0,0) -- (0,0.7);
    \draw[shift={(3.1,0)},rotate=116.5] (0,0) -- (0,4.5);

    %shooting manifold x=pi
    \draw [shift={(-4,-2)},domain=0:3.14,variable=\t,smooth] plot ({\t*cos(\t r)},{\t*sin(\t r)}); %upspiral
    \draw [shift={(-4,-2)},domain=-3.4:-3.14,variable=\t,smooth] plot ({\t},{-(0.7)*(\t-3.14)*(\t+3.14)}); %LHS
    \draw [shift={(-4,-2)},domain=0:3.14,variable=\t,smooth] plot ({-\t*cos(\t r)},{-\t*sin(\t r)}); %downsp
    \draw [shift={(-4,-2)},domain=3.14:3.4,variable=\t,smooth] plot ({\t},{(0.7)*(\t-3.14)*(\t+3.14)}); %RHS
    \filldraw (-5,-0.15) circle (0.1pt) node[anchor=south]{\footnotesize{$M_\pi$}};%label

    %intersection line
    \draw[very thick,shift={(-4,-2)}] (3,0) -- (-3,0) node[anchor=east]{\footnotesize{$L_\pi$}};
    
    %intersection points
    \filldraw (-0.9,-2) circle (3.5pt);    
    \filldraw (-4,-2) circle (3.5pt);    
    \filldraw (-7.1,-2) circle (3.5pt);    

    %from x=0 to 1
    \draw [domain=0:3.14,variable=\t,smooth] plot ({\t-5},{0.2*(cos(\t r))-0.35});%left
    \draw [domain=1.1:2.7,variable=\t,smooth] plot ({\t-2.6},{-1.2*(cos(\t r))-2.6});%right
    
    \end{tikzpicture}
    \caption{A shooting manifold with cross-section at $x=\pi$ given by the shooting curve $M_\pi$. % parametrized by $a\in L_0$. 
    The transverse intersections of $M_\pi$ and $L_\pi$ correspond to hyperbolic equilibria of \eqref{PDE}.% such that its Morse index and zero numbers can be computed by rotations of the tangent of $M_\pi$.
    }
\end{figure}
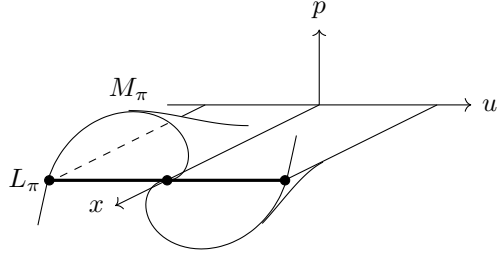
\begin{lem}\label{lem:eqbyshoot}
    \textbf{\emph{Hyperbolic Equilibria by Shooting.}} The set of equilibria $\mathcal{E}$ of \eqref{PDE} is in one-to-one correspondence with $M_{\pi}\cap L_{\pi}$.
    Moreover, an equilibrium $u_*\in\mathcal{E}$ %with boundary value $u_{*}(0)=a\in\mathbb{R}$ 
    is hyperbolic if, and only if, $M_\pi$ intersects $L_\pi$ transversely at $(\pi,u_*(\pi),0)$. %$(\pi,u_*(\pi,a,0),0)$.
\end{lem}
\begin{proof}
To prove the first part, note that a point in $M_{\pi}\cap L_{\pi}$ satisfies the equilibria equation with Neumann boundary conditions by definition of the shooting manifolds. Conversely, any equilibrium of \eqref{PDE} satisfying Neumann boundary must be in $M_{\pi}\cap L_{\pi}$. Due to the uniqueness of the differential equation \eqref{shootflow}, such correspondence above is one-to-one. The implicit function theorem guarantees these solutions solve \eqref{eqeq}.

To prove the second part, consider an equilibrium $u_*$ corresponding to the boundary value $u_*(0)=a\in\mathbb{R}$. We compare the eigenvalue problem for $u_*$ and the differential equation satisfied by the angle of the tangent vectors of the shooting manifold. 

The eigenvalue problem for $u_*$ is the linearization of the right hand side of \eqref{diffIFT2} at $u_*$
\begin{equation}\label{EVprobPRE}
    \lambda u = \tilde{F}_q(x,u_*,p_*,q_*).u_{xx}+\tilde{F}_p(x,u_*,p_*,q_*).u_x+\tilde{F}_u(x,u_*,p_*,q_*).u,
\end{equation}
where $x\in[0,\pi]$ has Neumann boundary conditions. From now on, we suppress the arguments of $\tilde{F}_q,\tilde{F}_p,\tilde{F}_u$.
We rewrite equation \eqref{EVprobPRE} as a first order system:
\begin{subequations}\label{EVprob}
\begin{align}
    u'&=p,\\
    p'&=-(\tilde{F}_p \cdot p+ \tilde{F}_u \cdot u-\lambda u)/\tilde{F}_q,\\
    x'&=1,
\end{align}
\end{subequations}
with Neumann boundary condition. This is well defined due to parabolicity $\tilde{F}_q\neq 0$.

On the other hand, $M_x$ is parametrized by $a\in\mathbb{R}$ and its tangent vector corresponding to the solution $u_*$ given by $(\partial_a u(x,a),\partial_a p(x,a))$ satisfies the linearized equation % of \eqref{shootflow}, %after deriving the shooting equation \eqref{shootflow} with respect to $u_0$
\begin{subequations}\label{tangentshootquasi}
\begin{align}
    u_{a}'&= p_{a},\\
	p_{a}'&=F^0_p\cdot p_{a}+F^0_u \cdot u_{a},\\
	x'&=1,
\end{align}
\end{subequations}
with data $ (u_{a}(0),p_{a}(0))=(1,0)$. We suppress the arguments $(u(x,a),p(x,a))$ of $F^0_u,F^0_p$.

Note the tangent equation in \eqref{tangentshootquasi} is the same as the eigenvalue problem \eqref{EVprob} with $\lambda=0$, except the tangent equation only has Neumann boundary conditions at $x=0$. Indeed, implicit differentiation of \eqref{PDE} with respect to $u$ (resp. $p$), bearing \eqref{diffIFT} in mind, yields $F_u=-f_u/f_q$ (resp. $F_p=-f_p/f_q$). %In particular, this holds in case that $u_t=0$. 
Similarly, $\tilde{F}_u=-f_u/f_r,\tilde{F}_p=-f_p/f_r$ and $\tilde{F}_q=-f_q/f_r$. Thus, $-\tilde{F}_u/\tilde{F}_q=F_u$ and $-\tilde{F}_p/\tilde{F}_q=F_p$, which implies \eqref{EVprob} becomes \eqref{tangentshootquasi}, noticing the argument $r_*=0$.%note that in 2.22 you must consider $\tilde{F}_r(r_*=0)r=0$, since $\tilde{F}$ does not depend on $r$. the argument $r_*=0$, in turn, appears when you change from $\tilde{F}$ to $F$, which then yield $F^0_u,F^0_p$.
To finish the proof, we compare equations \eqref{EVprob} and \eqref{tangentshootquasi} using well-known arguments, as in \cite[Lemma 2.7]{Lappicy18}.
%Since in such lemma, part 4) is not proved, we add the proof here for completeness.
\end{proof}
Thus, given a shooting curve $M_\pi$, one can find all equilibria of the PDE \eqref{PDE}. Next we address the characterization of adjacency (i.e., the Morse indices and zero numbers) of equilibria by means of the shooting curve, similar to \cite{Rocha85,Hale99,Rocha91,FiedlerRocha21}.

The \emph{Fusco-Rocha permutation} $\sigma$ is obtained by firstly labeling the intersection points $e_j\in M_{\pi}\cap L_\pi$ along the shooting curve $M_{\pi}$ following its parametrization, given by $\{(\pi,u(\pi,a,0),p(\pi,a,0))\}$ as $a\in\mathbb{R}$ increases, according to:
\begin{equation}
    e_1 \underset{\quad\, M_\pi}{\prec} ... \underset{\quad\, M_\pi}{\prec} e_N
\end{equation} 
where $N$ denotes the number of equilibria. Secondly, label the intersection points $e_j\in M_{\pi}\cap L_\pi$ along $L_\pi$ as $b\in\mathbb{R}$ increases according to
\begin{equation}
    e_{\sigma(1)} \,\, < \,\, ... \,\, < \,\, e_{\sigma(N)}.
\end{equation} 
Therefore $\sigma\in S_N$, the permutation group of $N$ elements. The Fusco-Rocha permutation $\sigma$ is enough to guarantee all information regarding adjacency, as in \cite{Rocha91,FiedlerRocha96}. 
See also \cite{FiedlerRocha21} for a more recent viewpoint.
%
%This can be proved using the third and fourth parts of Lemma \ref{lem:eqbyshoot}.%: the positive rotation along the shooting curve increases the Morse index and intersection of the shooting curve with a vertical line counts the zero numbers.
%\begin{lem}
%    \textbf{Equilibria Information Through Permutation.}
%    \begin{enumerate}
%    \item Morse indices
%    \item Zero numbers
%    \end{enumerate}
%\end{lem}
%
\begin{lem}\label{lem:adjbyshoot}
    \textbf{\emph{Adjacency Through Shooting.}}
    \begin{enumerate}
    \item  If the equilibrium $e_j\in\mathcal{E}$ %with boundary value $e_{j}(0)=a_j\in\mathbb{R}$ 
    is hyperbolic, then its Morse index is given by
    \begin{equation}
        i(e_j)%=\left\lfloor\frac{\theta(e_j)-\theta(e_0)}{\pi}\right\rfloor, 
        =1+\left\lfloor\frac{\theta(e_j)}{\pi}\right\rfloor,
    \end{equation}
    where $\theta(e_j)\in (-\pi/2,\infty)$ is the (clockwise) angle between the tangent of the curve $M_\pi$ and $L_\pi$ at their intersection $\{(\pi,e_j(\pi),0)\}$ and $\lfloor.\rfloor$ denotes the floor function.
    %where $\theta(e_j)$ denotes the (clockwise) angle of rotation of the tangent of the curve $M_\pi$ as $e_0(\pi)$ ranges to $e_j(\pi)$, and $\lfloor.\rfloor$ denotes the floor function.

    \item Consider $e_j,e_k\in \mathcal{E}$ with boundary value %$e_{j}(0)=a_j,e_{k}(0)=a_k$ and 
    $e_{j}(\pi)=b_j,e_{k}(\pi)=b_k$. Denote the vertical line at $b_j$ by $\ell_j:=\{(\pi,b_j,c) \textbf{ $|$ }  c\in\mathbb{R}\}$, and by $r_{jk}$ with $1\leq j < k\leq N$ the total number of intersections of $M_\pi$ and $\ell_j$ as $b$ ranges from $b_j$ to $b_k$, taking into account the sign of the rotation of the tangent of $M_\pi$ at said intersection (i.e., $+1$ in case of an intersection with a clockwise rotation and $-1$ for counter-clockwise). Then the number of intersection points of $e_j$ and $e_k$ can be computed as
    \begin{equation}\label{zeronumbers}
    z(e_j-e_k)=
    \begin{cases}
        i(e_j)+r_{jk}, \qquad\,\,\, \text{ if $(\theta_j \text{ mod } 2\pi)\in \left(\frac{\pi}{2},\pi\right) \cup \left(\frac{3\pi}{2},2\pi\right)$,}\\ %if $\theta_j$ belongs to odd quadrants
        i(e_j)-1+r_{jk}, \quad \text{ if $(\theta_j \text{ mod } 2\pi)\in \left(0,\frac{\pi}{2}\right)\cup \left(\pi,\frac{3\pi}{2}\right)$}. %if $\theta_j$ belongs to even quadrants
    \end{cases}
    \end{equation}
    \end{enumerate}
\end{lem}
The proof of Lemma \ref{lem:adjbyshoot} follows Lemma 2.7 of \cite{Lappicy18} and Proposition 3 of \cite{Rocha91}, bearing in mind that the fully nonlinear shooting equation \eqref{eqeq} is reduced to \eqref{shootflow} with linearized equation related to the tangent of the shooting curve, as in Lemma \ref{lem:eqbyshoot}.
\begin{figure}[H]
\minipage{0.33\textwidth}\centering
\begin{tikzpicture}[scale=0.25]
    \draw[->] (-7,0) -- (7,0) node[right] {\footnotesize{$u$}};
    \draw[->] (0,-5) -- (0,5) node[above] {\footnotesize{$p$}};

    %equilibria along L_0
    \filldraw [black] (-4.71,0) circle (5pt) node[anchor=north]{\footnotesize{\bm{$e_1$}}};
    \filldraw [black] (4.71,0) circle (5pt) node[anchor=north]{\footnotesize{\bm{$e_5$}}};
    \filldraw [black] (0,0) circle (5pt) node[anchor=north] {\footnotesize{\bm{$e_3$}}};
    \filldraw [black] (-1.57,0) circle (5pt) node[anchor=north]{\footnotesize{\bm{$e_2$}}};
    \filldraw [black] (1.57,0) circle (5pt) node[anchor=north]{\footnotesize{\bm{$e_4$}}};

    %equilibria along M_\pi
    \filldraw [black] (-4.71,0) circle (5pt) node[anchor=south]{\footnotesize{\bm{$e_1$}}};
    \filldraw [black] (4.71,0) circle (5pt) node[anchor=south]{\footnotesize{\bm{$e_5$}}};
    \filldraw [black] (0,0) circle (5pt) node[anchor=south]{\footnotesize{\bm{$e_3$}}};
    \filldraw [black] (-1.57,0) circle (5pt) node[anchor=south]{\footnotesize{\bm{$e_4$}}};
    \filldraw [black] (1.57,0) circle (5pt) node[anchor=south]{\footnotesize{\bm{$e_2$}}};
    
    %unstable    
    \draw [domain=0:4.71,variable=\t,smooth] plot ({\t*sin(\t r)},{-\t*cos(\t r)}); %upspiral
    \draw [domain=-5.2:-4.71,variable=\t,smooth] plot ({\t},{-(0.7)*(\t-4.71)*(\t+4.71)}); %LHS
    \draw [domain=0:4.71,variable=\t,smooth] plot ({-\t*sin(\t r)},{\t*cos(\t r)}); %downsp
    \draw [domain=4.71:5.2,variable=\t,smooth] plot ({\t},{(0.7)*(\t-4.71)*(\t+4.71)})node[anchor=south west]{\footnotesize{$M_\pi$}}; %RHS
\end{tikzpicture}
\endminipage\hfill 
\minipage{0.33\textwidth}\centering
\begin{tikzpicture}[scale=0.25]
    \draw[->] (-7,0) -- (7,0) node[right] {\footnotesize{$u$}};
    \draw[->] (0,-5) -- (0,5) node[above] {\footnotesize{$p$}};

    %equilibria along M_\pi
    \filldraw [black] (-4.71,0) circle (5pt) node[anchor=south]{\footnotesize{\bm{$e_1$}}};
    \filldraw [black] (4.71,0) circle (5pt) node[anchor=south]{\footnotesize{\bm{$e_5$}}};
    \filldraw [black] (0,0) circle (5pt) node[anchor=south]{\footnotesize{\bm{$e_3$}}};
    \filldraw [black] (-1.57,0) circle (5pt) node[anchor=south]{\footnotesize{\bm{$e_4$}}};
    \filldraw [black] (1.57,0) circle (5pt) node[anchor=south]{\footnotesize{\bm{$e_2$}}};

    %equilibria
    %\filldraw [black] (-4.71,0) circle (5pt);% node[anchor=north east]{$-1$};
    %\filldraw [black] (4.71,0) circle (5pt);% node[anchor=north west]{$+1$};
    %\filldraw [black] (0,0) circle (5pt);% node[anchor=north]{$0$};
    %\filldraw [black] (-1.57,0) circle (5pt);
    %\filldraw [black] (1.57,0) circle (5pt);% node[anchor=north west]{\footnotesize{$e_1$}};
    
    %unstable    
    \draw [domain=0:4.71,variable=\t,smooth] plot ({\t*sin(\t r)},{-\t*cos(\t r)}); %upspiral
    \draw [domain=-5.2:-4.71,variable=\t,smooth] plot ({\t},{-(0.7)*(\t-4.71)*(\t+4.71)}); %LHS
    \draw [domain=0:4.71,variable=\t,smooth] plot ({-\t*sin(\t r)},{\t*cos(\t r)}); %downsp
    \draw [domain=4.71:5.2,variable=\t,smooth] plot ({\t},{(0.7)*(\t-4.71)*(\t+4.71)}) node[anchor=south west]{\footnotesize{$M_\pi$}}; %RHS
    
    %tangent
    \draw[->, shift={(-4.75,0)}] (0,0) -- (1,2.6) node[above] {\footnotesize{$\theta(e_1)$}};
    \draw[->, shift={(1.65,0)},rotate=180] (0,0) -- (1,1.2) node[below] {\footnotesize{$\theta(e_2)$}};
\end{tikzpicture}
\endminipage\hfill 
\minipage{0.33\textwidth}\centering
\begin{tikzpicture}[scale=0.25]
    \draw[->] (-7,0) -- (7,0) node[right] {\footnotesize{$u$}};
    \draw[->] (0,-5) -- (0,5) node[above] {\footnotesize{$p$}};

    %equilibria
    \filldraw [black] (-4.71,0) circle (5pt);
    \filldraw [black] (4.71,0) circle (5pt) node[anchor=north west]{\footnotesize{$b_k$}};% node[anchor=north west]{$+1$};
    \filldraw [black] (0,0) circle (5pt);% node[anchor=north]{$0$};
    \filldraw [black] (-1.57,0) circle (5pt);
    \filldraw [black] (1.57,0) circle (5pt) node[anchor=north west]{\footnotesize{$b_j$}};
    
    %unstable    
    \draw [domain=0:4.71,variable=\t,smooth] plot ({\t*sin(\t r)},{-\t*cos(\t r)}); %upspiral
    \draw [domain=-5.2:-4.71,variable=\t,smooth] plot ({\t},{-(0.7)*(\t-4.71)*(\t+4.71)}); %LHS
    \draw [domain=4.71:5.2,variable=\t,smooth] plot ({\t},{(0.7)*(\t-4.71)*(\t+4.71)})node[anchor=south west]{\footnotesize{$M_\pi$}}; %RHS

    %bold segment
    \draw [very thick, domain=0:4.71,variable=\t,smooth] plot ({-\t*sin(\t r)},{\t*cos(\t r)}); %downsp
    \draw [very thick,domain=0:1.6,variable=\t,smooth] plot ({\t*sin(\t r)},{-\t*cos(\t r)}); %upspiral
    
    %aux line    
    \draw[dashed] (1.57,-5) -- (1.57,5) node[above] {\footnotesize{$\ell_j$}};
    
    %intersection
    \filldraw [black] (1.57,-3.25) circle (5pt);    
\end{tikzpicture}\vspace*{0.1cm}
\endminipage
\caption{On the left, the Fusco-Rocha permutation is obtained by consecutively labeling the equilibria along $M_\pi$ (top) and along $L_\pi$ (bottom), which yields the permutation $\sigma=(24)\in S_5$. In the middle, the angle $\theta (e_1)\in (-\pi/2,0)$ and thus the Morse index $i(e_1)=0$, whereas $i(e_j)$ increases by $+1$ (resp. -1) for each clockwise (resp. counter clockwise) $\pi$-rotation of $\theta(e_1)$; thereby $i(e_1)=0,i(e_2)=1,i(e_3)=2,i(e_4)=1,i(e_5)=0$. On the right, the number of (clockwise) intersections $r_{jk}=-1$ of the vertical line $\ell_j$ with the (bold) segment of $M_\pi$ ranging from $b_j$ to $b_k$. } \label{FIGX2}
\end{figure}
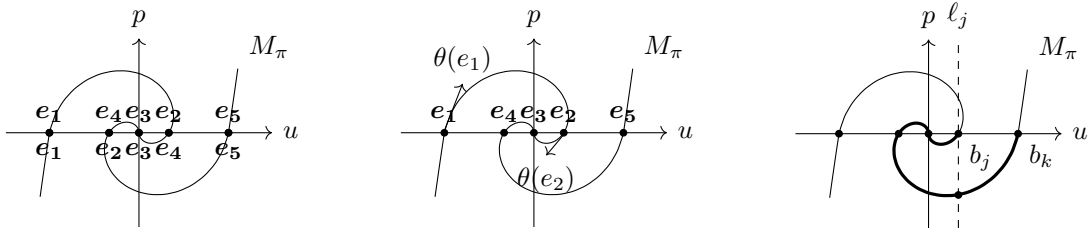

\section{Example: Fully nonlinear Chafee-Infante}\label{sec:CIfully}

In this section we provide the classical example of the Chafee-Infante global attractor that arises from the following fully nonlinear parabolic equation, 
\begin{equation}\label{PDECIfully}
    0=1-e^{u_t}+ u_{xx}+\lambda u[1-u^2]
\end{equation}
where $x\in[0,\pi]$ has Neumann boundary conditions and initial data $u_0\in C^{2\alpha+\beta}([0,\pi])$ with $\alpha,\beta \in (0,1)$ such that $\alpha \leq \text{min}\{1/2, 2\beta+1\}$. %, so that the equation generates a semiflow in such space, as in \cite{Lunardi95}.
For the semilinear version see \cite{Hale99}.

The equilibria equation describing the shooting curve is equal to the semilinear one,
\begin{subequations}\label{shootCIfully}
\begin{align}
    u'&=p,\\
    p'&=-\lambda u[1-u^2],\\
    x'&=1.
\end{align}
\end{subequations}
This yields the Chafee-Infante shooting curves as in Figure \ref{FIGshoot}. See \cite{FiedlerRocha96,LappicySing,FiedlerRocha21} and references therein.
Thus all equilibria, their respective Morse indices and zero numbers can be computed according to Lemma \ref{lem:adjbyshoot}. 
As an illustrative example, we compute the indices and intersections for the first parameter values below.
This detects which equilibria are adjacent and thereby which connections occur by means of a heteroclinic orbit yielding the same structure (connection-wise) as the standard Chafee-Infante problem. See Figure \ref{FIGCOR}.
\begin{figure}[H]
\minipage[b]{0.33\textwidth}\centering
\begin{subfigure}\centering
\begin{tikzpicture}[scale=1.2]
    \draw[->] (-1.5,0) -- (1.5,0) node[right] {$u$};
    \draw[->] (0,-1) -- (0,1) node[above] {$p$};
    
    %equilibria
    \filldraw [black] (-1,0) circle (1pt)  node[anchor=south]{\footnotesize{\bm{$e_1$}}};% node[anchor=north east]{$-1$};
    \filldraw [black] (1,0) circle (1pt)  node[anchor=south]{\footnotesize{\bm{$e_3$}}};% node[anchor=north west]{$+1$};
    \filldraw [black] (0,0) circle (1pt)  node[anchor=south]{\footnotesize{\bm{$e_2$}}};% node[anchor=north]{$0$};

    %unstable    
    \draw [domain=-1.2:1.2,variable=\t,smooth] plot ({\t},{\t*(\t-1)*(\t+1)});
    \filldraw [black] (1.2,0.7) circle (0.01pt) node[anchor= west]{$M_{\pi}$};     
\end{tikzpicture}\vspace*{0.075cm}
    \addtocounter{subfigure}{-1}\captionof{subfigure}{\footnotesize{$\lambda\in (\lambda_0,\lambda_1)$.}}\label{FIG:shootdim1}
\end{subfigure}
\endminipage\hfill
\minipage[b]{0.33\textwidth}\centering
\begin{subfigure}\centering
\begin{tikzpicture}[scale=0.25]
    \draw[->] (-7.5,0) -- (7.5,0) node[right] {$u$};
    \draw[->] (0,-5.1) -- (0,5.1) node[above] {$p$};
    
    %equilibria along M_\pi
    \filldraw [black] (-4.71,0) circle (5pt) node[anchor=south]{\footnotesize{\bm{$e_1$}}};
    \filldraw [black] (4.71,0) circle (5pt) node[anchor=south]{\footnotesize{\bm{$e_5$}}};
    \filldraw [black] (0,0) circle (5pt) node[anchor=south]{\footnotesize{\bm{$e_3$}}};
    \filldraw [black] (-1.57,0) circle (5pt) node[anchor=south]{\footnotesize{\bm{$e_4$}}};
    \filldraw [black] (1.57,0) circle (5pt) node[anchor=south]{\footnotesize{\bm{$e_2$}}};
    
    %equilibria
    \filldraw [black] (-4.71,0) circle (5pt);% node[anchor=north east]{$-1$};
    \filldraw [black] (4.71,0) circle (5pt);% node[anchor=north west]{$+1$};
    \filldraw [black] (0,0) circle (5pt);% node[anchor=north]{$0$};
    \filldraw [black] (-1.57,0) circle (5pt);
    \filldraw [black] (1.57,0) circle (5pt);
    
    %unstable    
    \draw [domain=0:4.71,variable=\t,smooth] plot ({\t*sin(\t r)},{-\t*cos(\t r)}); %upspiral
    \draw [domain=-5.2:-4.71,variable=\t,smooth] plot ({\t},{-(0.7)*(\t-4.71)*(\t+4.71)}); %LHS
    \draw [domain=0:4.71,variable=\t,smooth] plot ({-\t*sin(\t r)},{\t*cos(\t r)}); %downsp
    \draw [domain=4.71:5.2,variable=\t,smooth] plot ({\t},{(0.7)*(\t-4.71)*(\t+4.71)}) node[anchor= west]{$M_{\pi}$}; %RHS
    \end{tikzpicture}%\vspace*{0.2cm}
    \addtocounter{subfigure}{-1}\captionof{subfigure}{\footnotesize{$\lambda\in (\lambda_1,\lambda_2)$.}}\label{FIG:shootdim2}
\end{subfigure}
\endminipage\hfill
\minipage[b]{0.33\textwidth}\centering
\begin{subfigure}\centering
\begin{tikzpicture}[scale=0.155]
    \draw[->] (-10.5,0) -- (10.5,0) node[right] {$u$};
    \draw[->] (0,-8.5) -- (0,8.5) node[above] {$p$};
    
    %equilibria along M_\pi
    \filldraw [black] (-7.85,0) circle (8pt) node[anchor=south]{\footnotesize{\bm{$e_1$}}};% node[anchor=north east]{$-1$};
    \filldraw [black] (7.85,0) circle (8pt) node[anchor=south]{\footnotesize{\bm{$e_7$}}};% node[anchor=north west]{$+1$};        
    \filldraw [black] (-4.71,0) circle (8pt) node[anchor=south]{\footnotesize{\bm{$e_6$}}};
    \filldraw [black] (4.71,0) circle (8pt) node[anchor=south]{\footnotesize{\bm{$e_2$}}};
    \filldraw [black] (0,0) circle (8pt) node[anchor=north]{\footnotesize{\bm{$e_4$}}};
    \filldraw [black] (-1.57,0) circle (8pt) node[anchor=south]{\footnotesize{\bm{$e_3$}}};
    \filldraw [black] (1.57,0) circle (8pt) node[anchor=south]{\footnotesize{\bm{$e_5$}}};
        
    %unstable    
    \draw [domain=0:7.85,variable=\t,smooth] plot ({\t*sin(\t r)},{-\t*cos(\t r)}); %upspiral
    \draw [shift={(-3.14,0)},domain=-5.6:-4.71,variable=\t,smooth] plot ({\t},{-(0.7)*(\t-4.71)*(\t+4.71)}); %LHS
    \draw [domain=0:7.85,variable=\t,smooth] plot ({-\t*sin(\t r)},{\t*cos(\t r)}); %downsp
    \draw [shift={(3.14,0)},domain=4.71:5.6,variable=\t,smooth] plot ({\t},{(0.7)*(\t-4.71)*(\t+4.71)}) node[anchor= west]{$M_{\pi}$}; %RHS
    \end{tikzpicture}\vspace*{-0.05cm}
    \addtocounter{subfigure}{-1}\captionof{subfigure}{\footnotesize{$\lambda\in (\lambda_2,\lambda_3)$.}}\label{FIG:shootdim3}
\end{subfigure}
\endminipage
\caption{From left to right: the shooting curves for $\lambda$ in $(\lambda_0, \lambda_1)$, $(\lambda_1,\lambda_2)$, and $(\lambda_2,\lambda_3)$, respectively.} \label{FIGshoot}
\end{figure}
For $\lambda\in (\lambda_0,\lambda_1)$, the rotation of the tangent of $e_0$ yields $i(e_1)=0,i(e_2)=1,i(e_3)=0$. Moreover, the intersections of vertical lines $\ell_j$ at each equilibria $e_j$ with the shooting curves yield $r_{12}=r_{13}=r_{23}=0$, and the angles of the tangent vector satisfy $(\theta(e_1),\theta(e_3) \text{ mod } 2\pi) \in (\pi,3\pi/2)$, whereas $(\theta(e_2) \text{ mod } 2\pi)\in (0,\pi/2)$, we consequently obtain $z(e_1-e_2)=z(e_1-e_3)=z(e_2-e_3)=0$. Therefore $e_2$ is adjacent to both $e_1$ and $e_0$, since there are no equilibria between $e_2$ and either $e_1$ or $e_3$.

For $\lambda\in (\lambda_1,\lambda_2)$, then $i(e_1)=i(e_5)=0,i(e_2)=i(e_4)=1,i(e_3)=2$. Moreover, the $r_{1j}=0$ for all $j=2,...,5$, $r_{23}=r_{24}=0, r_{25}=-1$,  $r_{34}=0,r_{35}=-1$,  $r_{45}=-1$. Also $(\theta(e_1),\theta(e_5) \text{ mod } 2\pi) \in (\pi,3\pi/2)$, $(\theta(e_2),\theta(e_4) \text{ mod } 2\pi) \in (\pi/2,\pi)$ and $(\theta(e_3)\text{ mod } 2\pi) \in (\pi,3\pi/2)$.
Consequently,
\begin{subequations}\label{znumbCI}
\begin{align}
    z(e_1-e_j)&=0, \text{ for all $j=2,...,5$,} \\ 
    z(e_2-e_3)=z(e_2-e_4)&=1,\quad z(e_2-e_5)=0,\\
    z(e_3-e_4)&=1,\quad z(e_3-e_5)=0,\\
    z(e_4-e_5)&=0.
\end{align}
\end{subequations}
We only compute the heteroclinic connections between equilibria with Morse index differing by one, as the remaining connections follow by transitivity.
Indeed, $e_3$ is adjacent to $e_2,e_4$, as there are no equilibria between $e_3$ and either $e_2$ or $e_4$ at $x=0$; and $e_2,e_4$ are adjacent to $e_1,e_5$ since again there are no equilibria between each pair of equilibria. By transitivity, there are connections from $e_3$ to $e_1,e_5$.

Alternatively, one could compute adjacency among pairs of equilibria according to its definition \eqref{adjintro} and directly find heteroclinics between equilibria that do not necessarily have Morse indices differing by one. Indeed, $e_3$ is adjacent to $e_1,e_5$, since the only equilibrium between $e_1$ and $e_3$ at $x=0$ is $e_2$ (resp. the only equilibrium between $e_3$ and $e_5$ at $x=0$ is $e_4$), yet $z(e_1-e_3)=z(e_1-e_2)=0$ whereas $z(e_2-e_3)=1$. 

For $\lambda\in (\lambda_2,\lambda_3)$, then $i(e_1)=i(e_7)=0,i(e_2)=i(e_6)=1,i(e_3)=i(e_5)=2,i(e_4)=3$.  Moreover, the $r_{1j}=0$ for all $j=2,...,7$, $r_{23}=r_{24}=r_{25}=r_{26}=0, r_{27}=-1$,  $r_{34}=r_{35}=0,r_{36}=-1,r_{37}=-2$,  $r_{45}=0,r_{46}=-1,r_{47}=-2$, $r_{56}=-1, r_{57}=-2$, $r_{67}=-1$. Also $(\theta(e_1),\theta(e_3),\theta(e_5),\theta(e_7) \text{ mod } 2\pi) \in (\pi,3\pi/2)$, $(\theta(e_2),\theta(e_6) \text{ mod } 2\pi) \in (\pi/2,\pi)$ and $(\theta(e_4)\text{ mod } 2\pi) \in (0,\pi/2)$.
Thus,
\begin{subequations}\label{znumbCI2}
\begin{align}
    z(e_1-e_j)&=0, \text{ for all $j=2,...,7$,} \\ 
    z(e_2-e_j)&=1, \text{ for all $j=3,...,6$,}\qquad \, z(e_2-e_7)=0,\\
    z(e_3-e_4)=z(e_3-e_5)&=2,\quad\,\,\,\, z(e_3-e_6)=1, \qquad z(e_3-e_7)=0,\\
    z(e_4-e_5)&=2, \quad \,\,\,\, z(e_4-e_6)=1, \qquad z(e_4-e_7)=0,\\
    z(e_5-e_6)&=1, \quad \,\,\,\, z(e_5-e_7)=0,\\
    z(e_6-e_7)&=0.
\end{align}
\end{subequations}
This enables finding which equilibria are adjacent by analysing the equilibria $e_*$ that lie between $e_j$ and $e_k$ at $x=0$. 
We only compute the heteroclinic connections between equilibria with Morse index differing by one, as the remaining connections follow by transitivity.
Indeed, $e_4$ is adjacent to $e_3,e_5$, as there are no equilibria between $e_4$ and either $e_3$ or $e_5$ at $x=0$. Also, $e_3$ is adjacent to $e_2$, since there is no equilibria between $e_3$ and $e_2$ at $x=0$, and $e_3$ is adjacent to $e_6$ because the equilibria between them at $x=0$ are $e_4,e_5$, but $z(e_3-e_6)=1$ whereas $z(e_3-e_4)=z(e_3-e_5)=2$. Similarly, $e_5$ is adjacent to $e_2,e_6$. Analogously, $e_2$ is adjacent to $e_1$ since there is no equilibria between them at $x=0$, and $e_2$ is adjacent to $e_7$ since the equilibria between them at $x=0$ are $e_3,e_4,e_5,e_6$, but $z(e_2-e_7)=0$ whereas $z(e_2-e_j)=1$ for all $j=3,...,6$.
By similar reasoning, $e_6$ is adjacent to $e_1,e_7$.
\begin{figure}[H]
\minipage[b]{0.33\textwidth}\centering
\begin{subfigure}\centering
    \begin{tikzpicture}[scale=0.75]
    \draw[-] (-2,0) -- (2,0);
    \draw[->] (-1,0) -- (-1.01,0);
    \draw[->] (1,0) -- (1.01,0);
    
    \filldraw [thick] (-2,0) circle (2pt) node[left] {\footnotesize{$e_1$}};
    \filldraw [thick] (2,0) circle (2pt) node[right] {\footnotesize{$e_3$}};
    \filldraw [thick] (0,0) circle (2pt) node[above] {\footnotesize{$e_2$}};

    %anchor
    \filldraw [white,rotate=90,thick] (-2,0) circle (2pt) node[below] {\footnotesize{$-\Phi^\infty_1$}};
    \filldraw [white,rotate=90,thick] (2,0) circle (2pt) node[above] {\footnotesize{$+\Phi^\infty_1$}};   

    %OLD curved stuff
    %\draw [domain=0:6.28,variable=\t,smooth] plot ({\t},{0.3*(sin(\t r))});
    %\draw [->,domain=1.58:1.57,variable=\t,smooth] plot ({\t},{0.3*(sin(\t r))});
    %\draw [->,domain=4.72:4.73,variable=\t,smooth] plot ({\t},{0.3*(sin(\t r))});
    \end{tikzpicture}
    \addtocounter{subfigure}{-1}\captionof{subfigure}{\footnotesize{$\lambda\in (\lambda_0,\lambda_1)$.}}\label{FIG:dim1}
\end{subfigure}
\endminipage\hfill
\minipage[b]{0.33\textwidth}\centering

\begin{subfigure}\centering
    \begin{tikzpicture}[scale=0.75]
    %1d
    \draw[-] (-2,0) -- (2,0);
    \draw[->] (-1,0) -- (-1.01,0);
    \draw[->] (1,0) -- (1.01,0);
    
    %2d
    \draw[rotate=90, -] (-2,0) -- (2,0);
    \draw[rotate=90, ->] (-1,0) -- (-1.01,0);
    \draw[rotate=90, ->] (1,0) -- (1.01,0);
    
    \filldraw [rotate=90,thick] (-2,0) circle (2pt) node[below] {\footnotesize{$e_2$}};
    \filldraw [rotate=90,thick] (2,0) circle (2pt) node[above] {\footnotesize{$e_4$}};    
    
    %circle
    \draw (0,0) circle (57pt);% node[above] {\footnotesize{$u\equiv 0$}};

    %arrows circle
    \draw[<-] (1.2,1.6) arc (44:45:0.4cm and 0.4cm);    
    \draw[rotate=180,<-] (1.2,1.6) arc (44:45:0.4cm and 0.4cm);    

    \draw[rotate=70,->] (1.2,1.6) arc (44:45:0.4cm and 0.4cm);    
    \draw[rotate=70+180,->] (1.2,1.6) arc (44:45:0.4cm and 0.4cm);    

    %gominho
    \draw [domain=0:6.28,variable=\t,smooth] plot ({-2+0.64*\t},{0.9*(sin(\t r))});
    \draw[->] (-1,0.9) -- (-1.01,0.9);
    \draw[->] (1,-0.9) -- (1.01,-0.9);
    
    %reflected gominho
    \draw [domain=0:6.28,variable=\t,smooth] plot ({-2+0.64*\t},{-0.9*(sin(\t r))});
    \draw[->] (-1,-0.9) -- (-1.01,-0.9);
    \draw[->] (1,0.9) -- (1.01,0.9);
    
    %origin again
    \filldraw [thick] (-2,0) circle (2pt) node[left] {\footnotesize{$e_1$}};
    \filldraw [thick] (2,0) circle (2pt) node[right] {\footnotesize{$e_5$}};
    \filldraw [thick] (0,0) circle (2pt) node[above] {\footnotesize{$e_3$}};

    \end{tikzpicture}
    \addtocounter{subfigure}{-1}\captionof{subfigure}{\footnotesize{$\lambda\in (\lambda_1,\lambda_2)$.}}\label{FIG:dim2}
\end{subfigure}
\endminipage\hfill
\minipage[b]{0.33\textwidth}\centering
\begin{subfigure}\centering
    \begin{tikzpicture}[scale=0.75]
     %1d attractor
    \draw[dashed,-] (-2,0) -- (2,0);
    \draw[dashed,->] (-1,0) -- (-1.01,0);
    \draw[dashed,->] (1,0) -- (1.01,0);
    
    %2d
    \draw[rotate=90, dashed,-] (-2,0) -- (2,0);
    \draw[rotate=90, dashed,->] (-1,0) -- (-1.01,0);
    \draw[rotate=90, dashed,->] (1,0) -- (1.01,0);

    %3d
    \draw[dashed,-] (-0.4,-0.4) -- (0.4,0.4);
    \draw[->] (0.3,0.3) -- (0.31,0.31);
    \draw[rotate=180,->] (0.3,0.3) -- (0.31,0.31);
    
    %dots
    \filldraw [rotate=90,thick] (-2,0) circle (2pt) node[below] {\footnotesize{$e_2$}};
    \filldraw [rotate=90,thick] (2,0) circle (2pt) node[above] {\footnotesize{$e_6$}};     \filldraw [thick] (-2,0) circle (2pt) node[left] {\footnotesize{$e_1$}};
    \filldraw [thick] (2,0) circle (2pt) node[right] {\footnotesize{$e_7$}};
    \filldraw [thick] (0,0) circle (2pt) node[above] {\footnotesize{$e_4$}}; 
    \filldraw [thick] (-0.4,-0.4) circle (2pt) node[anchor=north east] {\footnotesize{$e_3$}};
    \filldraw [thick] (0.4,0.4) circle (2pt) node[anchor=south west] {\footnotesize{$e_5$}};        
    
    %circle
    \draw (0,0) circle (57pt);% node[above] {\footnotesize{$u\equiv 0$}};

    %arrows circle
    \draw[<-] (1.2,1.6) arc (44:45:0.4cm and 0.4cm);    
    \draw[rotate=180,<-] (1.2,1.6) arc (44:45:0.4cm and 0.4cm);    

    \draw[rotate=70,->] (1.2,1.6) arc (44:45:0.4cm and 0.4cm);    
    \draw[rotate=70+180,->] (1.2,1.6) arc (44:45:0.4cm and 0.4cm);       

    %equator
    \draw (-2,0) arc (180:360:2cm and 0.4cm);
    \draw[dashed] (-2,0) arc (180:0:2cm and 0.4cm);

    %greenwich
    \draw[rotate=-90] (-2,0) arc (180:360:2cm and 0.4cm);
    \draw[rotate=-90,dashed] (-2,0) arc (180:0:2cm and 0.4cm);
    
    %arrows greenwich+equator
    \draw[->] (0.34,1) arc (44:45:0.09cm and 0.4cm);    
    \draw[rotate=90,->] (0.34,1) arc (44:45:0.09cm and 0.4cm);    
    \draw[rotate=180,->] (0.34,1) arc (44:45:0.09cm and 0.4cm);    
    \draw[rotate=270,->] (0.34,1) arc (44:45:0.09cm and 0.4cm);    

    \draw[->] (-0.34,1) arc (131:130:0.09cm and 0.4cm);    
    \draw[rotate=90,->] (-0.34,1) arc (131:130:0.09cm and 0.4cm);    
    \draw[rotate=180,->] (-0.34,1) arc (131:130:0.09cm and 0.4cm);    
    \draw[rotate=270,->] (-0.34,1) arc (131:130:0.09cm and 0.4cm);    
    \end{tikzpicture}
    \addtocounter{subfigure}{-1}\captionof{subfigure}{\footnotesize{$\lambda\in (\lambda_2,\lambda_3)$.}} 
\end{subfigure}
\endminipage
\captionof{figure}{The Chafee-Infante attractor: dots correspond to equilibria and arrows to heteroclinics. 
}\label{FIGCOR}
\end{figure}
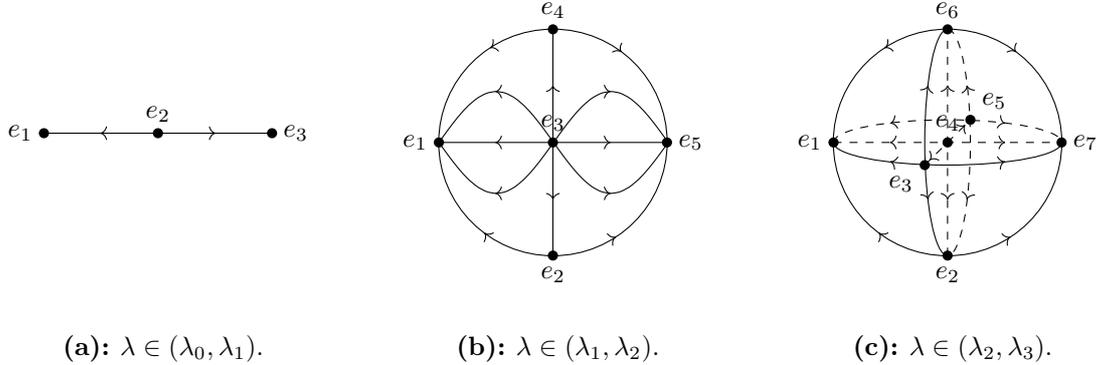

\section{Discussion}\label{sec:disc}

We now provide a discussion regarding the attractors of equation \eqref{PDE} towards a classification of PDEs by means of their global attractors. %, and particular cases when not all dependencies occur.  
As a preamble, let us fix Neumann boundary conditions, since attractors do not depend on them as long as the conditions are separated, as it was proved for the semilinear case in  \cite{Fiedler96}. We conjecture that this result still holds for fully nonlinear equations, under suitable hypothesis.
Moreover, we bear in mind that different PDEs \eqref{PDE} with the same Fusco-Rocha permutation yields $C^0$-orbit equivalence attractors. See \cite{FiedlerRocha00} for a proof in the semilinear setting, where the authors conjecture that this fact remains valid for the fully nonlinear setting.
Throughout the discussion, we consider $f\in C^2$.
As a final note, we mention that the task of finding a nonlinearity $f$ modelling an equation \eqref{PDE} that realizes a given Fusco-Rocha permutation was proved in \cite{FiedlerRocha99}.

Denote by $Sturm(u)$ the category with objects given by the global attractors (up to orbit equivalence) of the quasilinear equation with Hamiltonian type nonlinearities,
\begin{equation}\label{Hamiltonian}
    u_t=a(u)u_{xx}+f(u),
\end{equation}
where $a>0,a\in C^2$ satisfies the conditions \eqref{diss} that guarantee dissipativity. See \cite{EilenbergMacLane} for a classical introduction on category theory.

There are several morphisms in the $Sturm(u)$ category. Two that come to mind are bifurcation morphisms, namely $\mathcal{A}\mapsto p(\mathcal{A})$ if $p(\mathcal{A})$ arises from $\mathcal{A}$ after a pitchfork bifurcation; or $\mathcal{A}\mapsto sn(\mathcal{A})$ if $sn(\mathcal{A})$ is a result of $\mathcal{A}$ after a saddle-node bifurcation. There are other morphisms that emerge from topological constructions that respect the Sturm structure, such as a suspension of a well-known Sturm attractors that, see \cite{FiedlerRochatryp3}.

In \cite{FiedlerRochaWolfrum12}, it is given a characterization of the Sturm attractors of Hamiltonian type by means of the Fusco-Rocha permutation, except that the class $Sturm(u)$ consisted only of the reaction terms $f(u)$ with semilinear diffusion $a\equiv 1$. Nevertheless, the Hamiltonian quasilinear equations from \eqref{Hamiltonian} computed in \cite{Lappicy18} can be realized by semilinear ones, if one considers $u_t=u_{xx}+f(u)/a(u)$. Indeed, they have the same shooting equation, hence same permutation. Thus they should have the same attractor (up to orbit equivalence).

Consider also the class $Sturm(u,u_{xx},u_t)$, whose objects consist of global attractors of the following fully nonlinear equations of Hamiltonian type:
\begin{equation}\label{Hamiltonianfully}
    0=f(u,u_{xx},u_t),
\end{equation}
with parabolicity condition \eqref{par} such that the fully nonlinear equation \eqref{Hamiltonianfully} can be rewritten globally as $F^1(u,u_t)u_t=u_{xx}+F^0(u)$, where $F^0,F^1$ arise from $f$ through the implicit function theorem, as in \eqref{F0}-\eqref{F1}, and satisfy the dissipativity condition \eqref{diss}. 
Note that the global attractors of \eqref{Hamiltonianfully} can be realized by semilinear equations \eqref{Hamiltonian}. % , i.e., $Sturm(u,u_{xx},u_t)=Sturm(u)$. 
Indeed,
%Indeed, one can consider $F^0,F^1$ arising from $f$ through the implicit function theorem, as in \eqref{F0} and \eqref{F1}, and the fully nonlinear equation \eqref{Hamiltonianfully} transvestite in 
$F^1(u,u_t)u_t=u_{xx}+F^0(u)$ has the same shooting equation as the usual semilinear equation $u_t=u_{xx}+F^0(u)$. Hence they possess the same Fusco-Rocha permutation and thereby should have the same attractor (up to orbit equivalence). Thus,
\begin{equation}\label{STURM0}
    Sturm(u)=Sturm(u,u_{xx},u_t).
\end{equation}
Similarly, let $Sturm(x,u,u_x)$ be the category of attractors of \eqref{Hamiltonian} with advection dependent diffusion coefficient $a(x,u,u_x)$ and reaction $f(x,u,u_x)$. Such quasilinear attractors and the fully nonlinear attractors of \eqref{PDE}, denoted by $Sturm(x,u,u_x,u_{xx},u_t)$, can be realized by semilinear equations.
Indeed, %quasilinear equations contain the semilinear ones as the particular case that $a\equiv 1$, whereas 
a semilinear equation with reaction $f/a$ has the same Fusco-Rocha permutation of a quasilinear equation with diffusion coefficient $a>0$. % $u_t=au_{xx}+f$ with $a>0$, since they yield the same Fusco-Rocha permutation. 
Likewise, %fully nonlinear equations contain the quasilinear ones as a particular case, whereas 
any fully nonlinear equation \eqref{PDE} has a Fusco-Rocha permutation that can be realized by a semilinear equation with reaction term $F^0(x,u,u_x)$ emerging from the implicit function theorem in \eqref{F0}. Therefore,
\begin{equation}\label{STURM1}
    Sturm(x,u,u_x)=Sturm(x,u,u_x,u_{xx},u_t).
\end{equation}
It is known that there are Sturm attractors for advection dependent nonlinearities which can not be realized by Hamiltonian vector fields (see \cite{FiedlerRocha96,FiedlerRochaWolfrum12}), i.e.,
\begin{equation}\label{STURM2}
    Sturm(u) \subsetneq Sturm(x,u,u_x).
\end{equation}
Despite of all the nonlinear possibilities for $f$ in \eqref{PDE}, the dependencies on $u$ and $u_x$ play crucial a role in the complexity of the attractor. Even though there are no new attractors in the class of fully nonlinear parabolic equations, we enlarge the class of models that one is able to compute the Sturm attractors, and hence the domain of the Sturm and Fusco-Rocha functors, as the example in Section \ref{sec:CIfully}.
It remains the question of describing a full filtration diagram of classes of Sturm attractors such as in \eqref{STURM0}, \eqref{STURM1} and \eqref{STURM2}. For example, which Sturm attractors arise from the class of fully nonlinear diffusion in $Sturm(u_{xx},u_t)$? Where do the classes $Sturm(x), Sturm(x,u), Sturm(u_x)$ and $Sturm(u,u_{x})$ fit in such a diagram?
Can one give a permutation characterization to each of these classes, similar to Theorem 1 in \cite{FiedlerRochaWolfrum12} for Hamiltonian vector fields?

Next, denote by $S(u)$ the category of objects given by the Fusco-Rocha permutations within the group of permutations $S_n$, for all $n$, satisfying the conditions of Theorem 1 in \cite{FiedlerRochaWolfrum12} that classify the permutations of Hamiltonian vector fields. As before, there are morphisms in this category that also arise from bifurcations, such as $\sigma\mapsto p_*(\sigma)$ if $p_*(\sigma)$ is obtained from $\sigma$ after a pitchfork bifurcation and $\sigma\mapsto sn_*(\sigma)$ if $sn_*(\sigma)$ is achieved from $\sigma$ after a saddle-node bifurcation, and morphisms that arise from topological constructions such as a suspension. Note that all previous categories are graded by the dimension $n$ of the global attractor, for example $S(u)=S_n(u)\times \mathbb{N}_0$ where $S_n(u)$ are the Fusco-Rocha permutations in the group $S_n$ for Hamiltonian vector fields, or $Sturm(u)=Sturm_n(u)\times \mathbb{N}_0$ where $Sturm_n(u)$ are the Sturm attractors of dimension $n$.
Note that the bifurcation morphisms do not preserve the graded structure, as degree-2 maps, i.e., $p_*(\sigma)\in S_{n+2}$ for $\sigma \in S_n$.

The construction of Sturm attractors developed in the literature so far is a functor from the category of scalar parabolic equations of second order in one spatial variable, denoted by $parPDE(2,1)$, into the category $Sturm(x,u,u_x,u_{xx},u_t)$, and we call it the \emph{Sturm functor}, denoted by $Sturm: parPDE(2,1)\to Sturm(x,u,u_x,u_{xx},u_t)$. 
The Sturm functor is not injective with respect to the objects: two different parabolic PDEs can yield the same attractor (up to $C^0$-orbit equivalence) if they have the same Fusco-Rocha permutation.
The Sturm functor can be factorized by the functor from $parPDE(2,1)$ to its Fusco-Rocha permutation in $S(x,u,u_x,u_{xx},u_t)$, called \emph{Fusco-Rocha functor} which we denote by $\mathcal{F}\mathcal{R}:parPDE(2,1)\to S(x,u,u_x,u_{xx},u_t)$, and the functor $\mathcal{S}: S(x,u,u_x,u_{xx},u_t) \to Sturm(x,u,u_x,u_{xx},u_t)$ which constructs the attractor from a given permutation. In other words, $Sturm=\mathcal{F}\mathcal{R}\circ \mathcal{S}$.
%We are still far from an inverse type problem: given a global attractor, describe which class of PDEs (parabolic or hyperbolic, scalar or system, unidimensional or higher dimensional domain, etc) admit such global attractor?

The quest to explicitly describe and characterize the global attractors of partial differential equations % (and consequently the domain of the functor $Sturm=\mathcal{F}\mathcal{R}\circ \mathcal{S}$)
 goes on.

%\textcolor{red}{PL: talk about fully nonlinear boundary conditions. Need to prove: absense of super-exponential decay; see chapter 9 in lunardi} 

\textbf{Acknowledgment.} %I thank J. Nakasato, who fairly raised the concern that he never understood a paper of mine: this was certainly not his fault. This paper tries to bridge this gap. 
I was supported by FAPESP, 17/07882-0 and 18/18703-1.

\medskip

%\bibliographystyle{abbrv}%Used BibTeX style is unsrt
%\bibliography{mybib}
%\addcontentsline{toc}{section}{Bibliography}

\end{document}